\documentclass{article}

\usepackage{amssymb, amsfonts, amsmath, amsthm}
\usepackage[latin1]{inputenc}
\usepackage{url}
\usepackage{enumerate}
\usepackage{cancel}
\usepackage{color}
\usepackage{subfigure}

\author{Adrien Boussicault, Simone Rinaldi and Samanta Socci}
\title{The number of directed $k$-convex polyominoes}

\newtheorem{theorem}{Theorem}
\newtheorem{proposition}[theorem]{Proposition}
\newtheorem{lemma}[theorem]{Lemma}
\newtheorem{corollary}[theorem]{Corollary}

\newcommand{\tmppp}{parallelogram polyomino}
\newcommand{\pp}{\tmppp}
\newcommand{\pps}{\tmppp es}

\usepackage{tikz}

\newcommand{\SizeSquare}{1}
\newcommand{\LineWidth}{1}

\newcommand{\SizePoint}{.15*\SizeSquare}

\newcommand{\Point}[2]{(#1*\SizeSquare, #2*\SizeSquare)}
\newcommand{\AddSquare}[2]{
\draw[color=black, line width=\LineWidth] (#1*\SizeSquare-.5*\SizeSquare, #2*\SizeSquare-.5*\SizeSquare) -- (#1*\SizeSquare-.5*\SizeSquare, #2*\SizeSquare+.5*\SizeSquare) -- (#1*\SizeSquare+.5*\SizeSquare, #2*\SizeSquare+.5*\SizeSquare) -- (#1*\SizeSquare+.5*\SizeSquare, #2*\SizeSquare-.5*\SizeSquare) -- cycle;
}

\newcommand{\Ih}[4]{
\node[color=#4] at (#1*\SizeSquare, #2*\SizeSquare) {#3};
\filldraw[color=#4] ((#1*\SizeSquare+.5*\SizeSquare, #2*\SizeSquare+.5*\SizeSquare) circle (1.5*\SizePoint);
}

\makeatletter
\newcommand\xleftrightarrow[2][]{%
  \ext@arrow 9999{\longleftrightarrowfill@}{#1}{#2}}
\newcommand\longleftrightarrowfill@{%
  \arrowfill@\leftarrow\relbar\rightarrow}
\makeatother

\begin{document}
\maketitle

\section{Introduction}

In the plane $\mathbb{Z}\times \mathbb{Z}$ a \emph{cell} is a unit square and a
\emph{polyomino} is a finite connected union of cells. Polyominoes are defined up to translations.
Since they have been introduced by Golomb \cite{golomb}, polyominoes have become quite popular combinatorial objects and have shown relations with many mathematical problems, such as tilings \cite{BN}, or games \cite{gardner} among many others.

Two of the most relevant combinatorial problems concern the enumeration of polyominoes according to their {\em area} (i.e., number of cells) or {\em semi-perimeter}. These two problems are both difficult to solve and still open. As a matter of fact, the
number $a_n$ of polyominoes with $n$ cells is known up to $n = 56$ \cite{jensen} and asymptotically, these
numbers satisfy the relation $\lim _n (a_n)^{1/n} = \mu, 3.96<\mu<4.64$, where the lower bound is a recent
improvement of \cite{BMRR}.

In order to probe further, several subclasses of polyominoes have been introduced on which to
hone enumeration techniques.
Some of these subclasses can be defined using the notions of connectivity and directedness: among them we recall the convex, directed, parallelogram polyominoes, which will be considered in this paper.
Formal definitions of these classes will be given in the next section.

In the literature, these objects have been widely studied using different techniques. Here, we just outline some results which will be useful for the reader of this paper:
\begin{description}
\item{(i)} The number $f_n$ of convex polyominoes with semi-perimeter $n \geq 2$ was obtained by Delest and Viennot in \cite{DV}, and it is:
$$f_{n+2} = (2n + 11)4^n - 4(2n + 1)\binom{2n}{n}, \qquad n \geq 0; \qquad f_0 = 1, \qquad f_1 = 2.$$
\item{(ii)} The number of directed convex polyominoes with semi-perimeter $n+2$ is
equal to ${{2n} \choose n}$ \cite{BousquetMelou96amethod,Bousquet1992,LinChang88}.
\item{(iii)} The number of parallelogram polyominoes with semi-perimeter $n+1$ is the $n$th {\em Catalan number} \cite{DV}.
\end{description}

\medskip

In this paper we present a new general approach for the enumeration of directed convex
polyominoes, which let us easily control several statistics.
This method relies on a bijection between directed convex
polyominoes with semi-perimeter equal to $n+2$ and triplets $(F_e,F_s,\lambda)$, where $F_e$ and $F_s$
are forests of $m_e$ and $m_s$ trees, respectively, with a total number of nodes equal to $n$, and $\lambda$ is a lattice path made of $m_e+1$ east and $m_s+1$ south unit steps (see Proposition \ref{prop:directed}).

Basing on this bijection, we develop a new method for the enumeration of directed convex polyominoes, according to several different parameters, including the semi-perimeter, the degree of convexity, the width, the height, the size of the last row/column and the number of corners. We point out that most of these statistics have already been considered in the literature (see, for instance \cite{BFR,BousquetMelou96amethod,LinChang88}), but what makes our method interesting, to our opinion, is that every statistic which can be read on the two forests $F_e$ and $F_s$, can be in principle computed. Basically, with $\cal C$ being a class of directed-convex polyominoes, our method consists of three steps:

\begin{enumerate}
\item provide a characterization of the two forests $F_e,F_s$ and of the path $\lambda$, for the polyominoes in $\cal C$;
\item determine the generating function -- according to the considered statistics -- of the forests $F_e,F_s$ and of the paths $\lambda$, for the polyominoes in $\cal C$;
\item obtain the generating function of $\cal C$ by means of the composition of the generating function of the paths $\lambda$ with the generating functions of the trees of $F_e,F_s$.
\end{enumerate}

Furthermore, the previously described bijection can be easily translated into a (new) bijection between directed convex polyominoes and Grand-Dyck paths (or bilateral Dyck paths). Other bijections between these two classes of objects can be found in the literature, see for instance, \cite{BFR} and \cite{Bousquet1992}.

\medskip

The most important and original result of the paper consists in applying our method to the enumeration of {\em directed $k$-convex polyominoes}, i.e. directed convex polyominoes which are also $k$-convex. Let us recall that in \cite{Castiglione2003290} it was proposed a classification of convex polyominoes based on the number of changes of direction in the paths connecting any two cells of a polyomino. More
precisely, a convex polyomino is {\itshape $k$-convex} if every pair of its cells can be connected by a monotone path
with at most $k$ changes of direction.

For $k=1$ we have the $L$-convex polyominoes, where any two cells can be connected by a path with at most one change of direction. In the recent literature $L$-convex polyominoes have been considered from several points of view: in \cite{Castiglione2003290} it is shown that they are a well-ordering according to the sub-picture order; in \cite{CFRR} the authors have investigated some tomographical aspects, and have discovered that $L$-convex polyominoes are uniquely determined by their horizontal and vertical projections. Finally, in \cite{lconv,CFRR2} it is proved that the number $l_n$ of $L$-convex polyominoes having semi-perimeter equal to $n + 2$ satisfies the recurrence relation
$$l_{n+2} = 4 l_{n+1} - 2 l_{n}, \qquad n\geq 1, \qquad l_0 = 1, \qquad l_1 = 2 \qquad  l_2 = 7.$$
For $k=2$ we have $2$-convex (or $Z$-convex) polyominoes, such that each two cells can be connected by a path with at most two changes of direction. Unfortunately, $Z$-convex polyominoes do not inherit most of the combinatorial properties of $L$-convex polyominoes. In particular, their enumeration resisted standard enumeration techniques and it was obtained in \cite{DRS} by applying the so-called {\em inflation method}.
The authors proved that the generating function of $Z$-convex polyominoes with
respect to the semi-perimeter is:
$$
\frac{2z^4(1-2z)^2d(z)}
{(1-4z)^2(1-3z)(1-z)}+\frac{z^2(1-6z+10z^2-2z^3-z^4)}
{(1-4z)(1-3z)(1-z)},
$$
where $
d(z)=\frac{1}{2} \left( 1-2z-\sqrt{1-4z} \right)
$
is the generating function of Catalan numbers \cite{DRS}. Hence the number of Z-convex polyominoes having semi-perimeter $n+2$ grows asymptotically as $\frac{n}{24} 4^n$.
However, the solution found for $2$-convex polyominoes seems to be not easily generalizable to a generic $k$, hence the problem of enumerating $k$-convex polyominoes for $k>2$ is still open and difficult to solve. Recently,
some efforts in the study of the asymptotic behavior of $k$-convex polyominoes have been made by Micheli and Rossin in \cite{MR}.

Thus, in order to simplify the problem of enumerating polyominoes satisfying the $k$-convexity constraint, the authors of \cite{BatFedRinSoc} provide a general method to enumerate an important subclass of $k$-convex polyominoes, that of $k$-parallelogram
polyominoes, i.e. $k$-convex polyominoes which are also parallelogram. In this case, the authors prove that the generating function is rational for any $k\geq 0$, and it can be expressed as follows:
\begin{equation}\label{eq:kparal}
z^2\left( \frac{F_{k+2}}{F_{k+3}} \right)^2 - z^2
\left( \frac{F_{k+2}}{F_{k+3}} - \frac{F_{k+1}}{F_{k+2}} \right)^2,
\end{equation}
where $F_k$ are the \emph{Fibonacci polynomials} are defined by the following recurrence relation
$$
F_0 = 0, \hspace{.5cm} F_1 = 1 \hspace{.5cm} \mbox{ and } \hspace{.5cm} F_{k+2} = F_{k+1} - z F_{k} \hspace{.5cm} \text{ for } \hspace{.5cm} k\ge0.
$$
At the end of \cite{BatFedRinSoc}  the authors observe that in \cite{BruKnuRic} it is shown that the generating function of trees with height less than or equal to $k$ can be also expressed in terms of the Fibonacci polynomials, and it is precisely equal to
\begin{equation}\label{eq:parallelogram}
\sum_{j\ge0} t_{j,\le k} z^j = \frac{z F_{k}}{F_{k+1}}.
\end{equation}
Thus, considering (\ref{eq:parallelogram}), it was proposed the problem of finding a bijection between $k$-parallelogram polyominoes and pairs of trees having height at most $k+2$, minus pairs of trees having both height exactly equal to $k+2$.

\medskip

As already mentioned, in this paper we succeed in solving this problem, as a specific sub-case of a more general framework: indeed we are able to provide a bijective proof for the number of directed $k$-convex polyominoes. In order to reach this goal, we rely on the observation that our bijection can be used to express the convexity degree of the polyomino in terms of the heights of the trees of the forests $F_e$ and $F_s$.
Basing on this observation, the application of our method eventually let us determine the generating function of directed $k$-convex polyominoes, for any $k\geq 1$. This is a rational function, which be suitably expressed, again, in terms of the
Fibonacci polynomials. More precisely, we prove that, for every $k\geq 1$, the generating function of directed $k$-convex polyominoes according to the semi-perimeter is equal to:
\begin{equation}\label{eq:kdirconv}
z^2\,  \left( \frac{F_{k+2}}{F_{2k+3}} \right)^2\,  F_{2k+2}.
\end{equation}
We point out that (\ref{eq:kdirconv}) is the first result in the literature concerning the $k$-convexity constraint on directed-convex polyominoes.  Due to the rationality of the generating function we can then apply standard techniques to provide the asymptotic behavior for the number of directed $k$-convex polyominoes.

Moreover, if we restrict ourselves to the case of parallelogram polyominoes, we observe that our bijection reduces to a bijection between parallelogram polyominoes with semi-perimeter and pairs of trees, and this gives a simple tool to present a combinatorial proof for (\ref{eq:kparal}).

\section{Notation and preliminaries}
In this section we recall some basic definitions about polyominoes, and other combinatorial objects, which will be used in the rest of the paper. 
First, we point out that, in the representaion of polyominoes, we use the convention that the south-west corner of the minimal bounding rectangle is placed at $(0,0)$. A \emph{column} (\emph{row}) of a polyomino is the intersection of the polyomino and an infinite strip of cells whose centers lie on a vertical (horizontal) line. A polyomino is said to be \emph{column-convex} (\emph{row-convex}) when its intersection with any vertical (horizontal) line is connected. Figure~\ref{polyominoes}$(b)$, $(c)$ show a column convex and a row convex polyomino, respectively. A polyomino is \emph{convex} if it is both column and row convex (see Fig.~\ref{polyominoes}$(d)$).
The \emph{semi-perimeter} of a polyomino is naturally defined as half the perimeter, while the \emph{area} is the number of its cells.
As a matter of fact, the semi-perimeter of a convex polyomino is given by the sum of the number of its rows and its columns.

\begin{figure}[htb]
\begin{center}
\includegraphics[width=.7\textwidth]{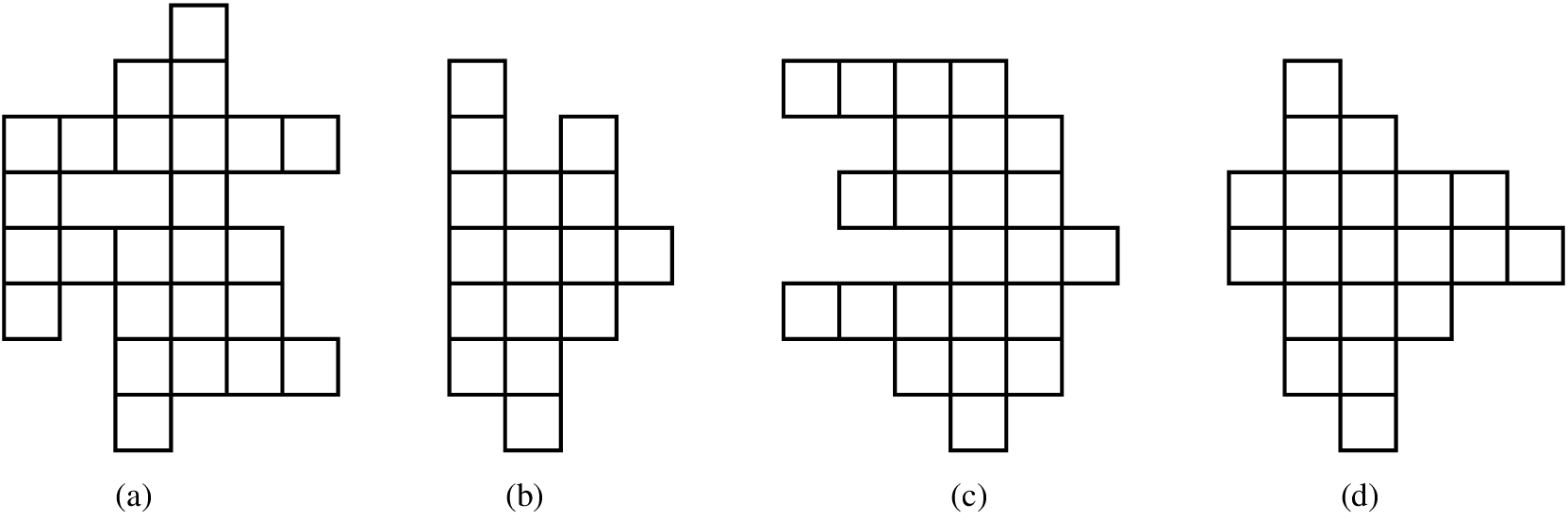}
\caption{Examples of polyominoes. \label{polyominoes}}
\end{center}
\end{figure}

Let $P$ be a convex polyomino whose minimal bounding rectangle has dimension 
$l_1 \times l_2$. We number the $l_1$ columns and the $l_2$ rows from left to 
right and from bottom to top, respectively. 
Thus, we consider the bottom 
(resp. top) row of $P$ as its first (resp. last) row, and the leftmost 
(resp. rightmost) column of $P$ as its first (resp. last) column.
By convention, we will often write $(i,j)$ to denote the cell of $P$, whose 
north-west corner has coordinates equal to $(i,j)$.
With that convention, the cell is at the intersection of the $i$-th column and 
the $j$-th row.

A \emph{path} is a self-avoiding sequence of unit steps of four types: north 
$n=(0,1)$, south $s=(0,-1)$, east $e=(1,0)$, and west $w=(-1,0)$. 
A path is said to be \emph{monotone} if it is comprised only of steps of two
types. 

A polyomino is said to be \emph{directed} when each of its cells can be reached from a distinguished cell, called the root (denoted by $S$), by a path which is contained in the polyomino and uses only north and east unit steps. 
A polyomino is \emph{directed convex} if it is both directed and convex. 

We recall that a \emph{\pp} is a polyomino whose boundary can be decomposed in two paths, the upper and the lower paths, which are comprised of north and east unit steps and meet only at their starting and final points. 
It is clear that parallelogram polyominoes are also directed convex polyominoes, while the converse does not hold.

Figure~\ref{parallelograms} shows all directed convex polyominoes with semi-perimeter equal to $4$, where only the rightmost polyomino is not a parallelogram one.

\begin{figure}[htb]
\begin{center}
\includegraphics[width=.5\textwidth]{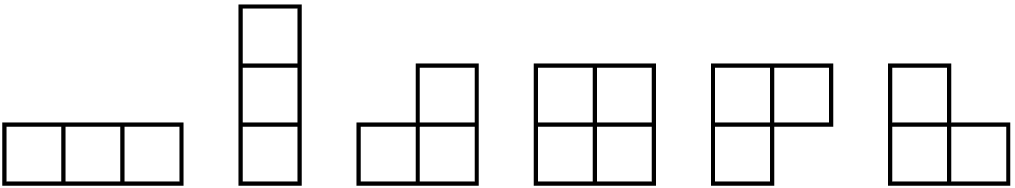}
\end{center}
\caption{The six directed convex polyominoes with semi-perimeter $4$. \label{parallelograms}}
\end{figure}

Now, we introduce another class of objects, which will be useful in this paper.
An \emph{ordered tree} is a rooted tree for which an ordering is specified for the children of each vertex. In this paper we shall say simply \emph{tree} instead of \emph{ordered tree}. The \emph{size} of a tree is the number of nodes.
The \emph{height} $h(T)$ of a tree $T$ is the number of nodes on a maximal simple path starting at the root of $T$.

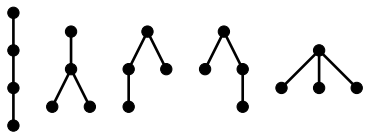
\begin{figure}[!ht]
$$
\begin{array}{c}
\begin{tikzpicture}[scale=.5]
\draw[color=black, line width=\LineWidth] (0,0) -- (0, -1);
\draw[color=black, line width=\LineWidth] (0,-1) -- (0, -2);
\draw[color=black, line width=\LineWidth] (0,-2) -- (0, -3);

\foreach \x/\y in { 0/0, 0/-1, 0/-2, 0/-3 }{
    \filldraw[color=black] (\x, \y) circle (\SizePoint);
}
\end{tikzpicture}
\end{array}
\begin{array}{c}
\begin{tikzpicture}[scale=.5]
\draw[color=black, line width=\LineWidth] (0,0) -- (0, -1);
\draw[color=black, line width=\LineWidth] (0,-1) -- (-0.5, -2);
\draw[color=black, line width=\LineWidth] (0,-1) -- (0.5, -2);

\foreach \x/\y in { 0/0, 0/-1, 0.5/-2, -0.5/-2 }{
    \filldraw[color=black] (\x, \y) circle (\SizePoint);
}
\end{tikzpicture}
\end{array}
\begin{array}{c}
\begin{tikzpicture}[scale=.5]
\draw[color=black, line width=\LineWidth] (-0.5,-1) -- (-0.5, -2);
\draw[color=black, line width=\LineWidth] (0,0) -- (-0.5, -1);
\draw[color=black, line width=\LineWidth] (0,0) -- (0.5, -1);

\foreach \x/\y in { 0/0, 0.5/-1, -0.5/-1, -0.5/-2 }{
    \filldraw[color=black] (\x, \y) circle (\SizePoint);
}
\end{tikzpicture}
\end{array}
\begin{array}{c}
\begin{tikzpicture}[scale=.5]
\draw[color=black, line width=\LineWidth] (0.5,-1) -- (0.5, -2);
\draw[color=black, line width=\LineWidth] (0,0) -- (-0.5, -1);
\draw[color=black, line width=\LineWidth] (0,0) -- (0.5, -1);

\foreach \x/\y in { 0/0, 0.5/-1, -0.5/-1, 0.5/-2 }{
    \filldraw[color=black] (\x, \y) circle (\SizePoint);
}
\end{tikzpicture}
\end{array}
\begin{array}{c}
\begin{tikzpicture}[scale=.5]
\draw[color=black, line width=\LineWidth] (0,0) -- (-1, -1);
\draw[color=black, line width=\LineWidth] (0,0) -- (0, -1);
\draw[color=black, line width=\LineWidth] (0,0) -- (1, -1);

\foreach \x/\y in { 0/0, 1/-1, 0/-1, -1/-1 }{
    \filldraw[color=black] (\x, \y) circle (\SizePoint);
}
\end{tikzpicture}
\end{array}
$$
\caption{Ordered trees of size $4$. \label{fig:ordered_tree_size_3}}
\end{figure}

All trees with size equal to $4$ are shown in Fig.~\ref{fig:ordered_tree_size_3}, while  Figure~\ref{fig:ordered_tree_heigt_4} depicts some trees with height equal to $4$.

\begin{figure}[!ht]
$$
\begin{array}{c}
\begin{tikzpicture}[scale=.5]
\draw[color=black, line width=\LineWidth] (0,0) -- (0, -1);
\draw[color=black, line width=\LineWidth] (0,-1) -- (0, -2);
\draw[color=black, line width=\LineWidth] (0,-2) -- (0, -3);

\foreach \x/\y in { 0/0, 0/-1, 0/-2, 0/-3 }{
    \filldraw[color=black] (\x, \y) circle (\SizePoint);
}
\end{tikzpicture}
\end{array}
\begin{array}{c}
\begin{tikzpicture}[scale=.5]
\draw[color=black, line width=\LineWidth] (0,1) -- (0, 0);
\draw[color=black, line width=\LineWidth] (-0.5,-1) -- (-0.5, -2);
\draw[color=black, line width=\LineWidth] (0,0) -- (-0.5, -1);
\draw[color=black, line width=\LineWidth] (0,0) -- (0.5, -1);

\foreach \x/\y in { 0/1, 0/0, 0.5/-1, -0.5/-1, -0.5/-2 }{
    \filldraw[color=black] (\x, \y) circle (\SizePoint);
}
\end{tikzpicture}
\end{array}
\begin{array}{c}
\begin{tikzpicture}[scale=.5]
\draw[color=black, line width=\LineWidth] (0,0) -- (-1, -1);
\draw[color=black, line width=\LineWidth] (0,0) -- (0, -1);
\draw[color=black, line width=\LineWidth] (0,0) -- (1, -1);
\draw[color=black, line width=\LineWidth] (-1,-1) -- (-1, -2);
\draw[color=black, line width=\LineWidth] (1,-1) -- (1, -2);
\draw[color=black, line width=\LineWidth] (1,-2) -- (0.5, -3);
\draw[color=black, line width=\LineWidth] (1,-2) -- (1.5, -3);

\foreach \x/\y in { 0/0, -1/-1, 0/-1, 1/-1, -1/-1, -1/-2, 1/-2, .5/-3, 1.5/-3 }{
    \filldraw[color=black] (\x, \y) circle (\SizePoint);
}
\end{tikzpicture}
\end{array}
$$
\caption{Examples of ordered trees with height $4$ \label{fig:ordered_tree_heigt_4}}
\end{figure}
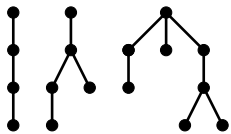

Now, let $t_{n,\le k}$ be the number of trees of size $n$ and height less than or equal to $k$.
In \cite{BruKnuRic} the generating function $\mathcal{T}_{\le k}$ of trees with height less than or equal to $k$ is given by N.G. de Bruijn, D.E. Knuth and S.O. 
Rice:
$$
\mathcal{T}_{\le k} = \sum_{j\ge0} t_{j,\le k} z^j = \frac{z F_{k}}{F_{k+1}}
$$
where $F_{k}$ are the Fibonacci polynomials.
The \emph{Fibonacci polynomials} are defined by the following recurrence relation:
$$
F_0 = 0, \hspace{.5cm} F_1 = 1 \hspace{.5cm} \mbox{ and } \hspace{.5cm} F_{k+2} = F_{k+1} - z F_{k} \hspace{.5cm} \text{ for } \hspace{.5cm} k\ge0.
$$

This result can be simply explained by observing that a tree with height less than or equal to $k+1$ can be obtained by taking a root node and attaching zero or more subtrees each of which has height less than or equal to $k$.
Hence, $$\mathcal{T}_{\leq k+1} = \sum_{j\ge 0} z ( \mathcal{T}_{\leq k} )^j = \frac{z}{1- \mathcal{T}_{\leq k}} .$$
Now, we can  easily check that $\mathcal{T}_{\leq k} = \frac{z F_k}{F_{k+1}}$ is the correct solution.

Since $F_k$ is a sequence of the first order, it is easy to obtain the formula:
$$
F_k \, = \, \frac{1}{\sqrt{1-4z}} 
\left(
	\left(
		\frac{ 1+\sqrt{1-4z} }{ 2 }
	\right)^k
	-
	\left(
		\frac{ 1-\sqrt{1-4z} }{ 2 }
	\right)^k
\right).
$$

We define by $\mathcal{T}_{=k}$ the generating function of trees 
with height equal to $k$.
Easily we can obtain $\mathcal{T}_{=k}$ as the difference between the generating function $\mathcal{T}_{\leq k}$ of the class of trees with height less than or equal to $k$ and the generating function $\mathcal{T}_{\leq k-1}$ of the class of trees having height less than or equal to $k-1$; namely, $\mathcal{T}_{=k} = \mathcal{T}_{\leq k} - \mathcal{T}_{\leq k-1}$.

It is possible to obtain $\mathcal{T}_{=k}$ in a different way:
let  $T$ be a tree having height equal to $k$, and denote by $e$ the rightmost node of $T$ having height $k$ with respect to the root of $T$. We consider the maximal chain starting at $e$ and ending at the root of $T$.
Now each node of this chain (different from the node $e$) having height equal to $k-i+1$ (with $i=2,\ldots k$) can be considered the root of two trees having heights $i$ and $i-1$, respectively (see Figure~\ref{fig:tree_height_k}). This argument leads to a bijection between trees with height equal to $k$ and the sequence $(T_1,T_2,\ldots,T_{2(k-1)})$ of trees such that $T_{i}$ and $T_{i+1}$ are trees having height $i$ and $i-1$, respectively. Hence, 
$$
\mathcal{T}_{=k}
=
z\prod_{i=2}^k \mathcal{T}_{\leq i} \mathcal{T}_{\leq i-1} z^{-1}
=
z\prod_{i=2}^{k} z \frac{\cancel{F_i}}{F_{i+1}} \frac{F_{i-1}}{\cancel{F_i}}
=
\frac{z^k}{F_{k}F_{k+1}}.
$$

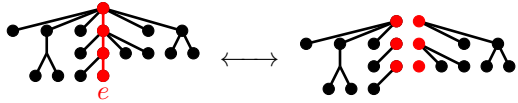
\begin{figure}[!ht]
\centering 
$
\begin{array}{c}
\begin{tikzpicture}[scale=.5]
\newcommand{\sizeGridTree}{.3}
\draw[color=black, line width=\LineWidth] (-3*\sizeGridTree,4*\sizeGridTree) -- (5*\sizeGridTree, 6*\sizeGridTree);
\draw[color=black, line width=\LineWidth] (-1*\sizeGridTree,0*\sizeGridTree) -- (0*\sizeGridTree, 2*\sizeGridTree);
\draw[color=black, line width=\LineWidth] (1*\sizeGridTree,0*\sizeGridTree) -- (0*\sizeGridTree,2*\sizeGridTree);
\draw[color=black, line width=\LineWidth] (3*\sizeGridTree,0*\sizeGridTree) -- (5*\sizeGridTree,2*\sizeGridTree);
\draw[color=black, line width=\LineWidth] (0*\sizeGridTree,2*\sizeGridTree) -- (0*\sizeGridTree,4*\sizeGridTree);
\draw[color=black, line width=\LineWidth] (3*\sizeGridTree,2*\sizeGridTree) -- (5*\sizeGridTree,4*\sizeGridTree);
\draw[color=black, line width=\LineWidth] (7*\sizeGridTree,2*\sizeGridTree) -- (5*\sizeGridTree,4*\sizeGridTree);
\draw[color=black, line width=\LineWidth] (9*\sizeGridTree,2*\sizeGridTree) -- (5*\sizeGridTree,4*\sizeGridTree);
\draw[color=black, line width=\LineWidth] (11*\sizeGridTree,2*\sizeGridTree) -- (12*\sizeGridTree,4*\sizeGridTree);
\draw[color=black, line width=\LineWidth] (13*\sizeGridTree,2*\sizeGridTree) -- (12*\sizeGridTree,4*\sizeGridTree);
\draw[color=black, line width=\LineWidth] (0*\sizeGridTree,4*\sizeGridTree) -- (5*\sizeGridTree,6*\sizeGridTree);
\draw[color=black, line width=\LineWidth] (3*\sizeGridTree,4*\sizeGridTree) -- (5*\sizeGridTree,6*\sizeGridTree);
\draw[color=black, line width=\LineWidth] (9*\sizeGridTree,4*\sizeGridTree) -- (5*\sizeGridTree,6*\sizeGridTree);
\draw[color=black, line width=\LineWidth] (12*\sizeGridTree,4*\sizeGridTree) -- (5*\sizeGridTree,6*\sizeGridTree);

\foreach \x/\y in { -1/0, 1/0, 3/0, 5/0, 3/2, 5/2, 7/2, 9/2, 11/2, 13/2, 0/4, 3/4, 5/4, 9/4, 12/4, 5/6, -3/4 }{
    \filldraw[color=black] (\x*\sizeGridTree, \y*\sizeGridTree) circle (\SizePoint);
}

\draw[color=red, line width=\LineWidth] (5*\sizeGridTree,0*\sizeGridTree) -- (5*\sizeGridTree,2*\sizeGridTree);
\draw[color=red, line width=\LineWidth] (5*\sizeGridTree,2*\sizeGridTree) -- (5*\sizeGridTree,4*\sizeGridTree);
\draw[color=red, line width=\LineWidth] (5*\sizeGridTree,4*\sizeGridTree) -- (5*\sizeGridTree,6*\sizeGridTree);

\foreach \x/\y in { 5/0, 5/2, 5/4, 5/6 }{
    \filldraw[color=red] (\x*\sizeGridTree, \y*\sizeGridTree) circle (\SizePoint);
}

\node[color=red, below=1pt] at (5*\sizeGridTree,0*\sizeGridTree) {$e$}; 
\end{tikzpicture}
\end{array}
\xleftrightarrow{\hspace{.5cm}}{}
\begin{array}{c}
\begin{tikzpicture}[scale=.5]
\newcommand{\sizeGridTree}{.3}
\draw[color=black, line width=\LineWidth] (-3*\sizeGridTree,4*\sizeGridTree) -- (5*\sizeGridTree, 6*\sizeGridTree);
\draw[color=black, line width=\LineWidth] (-1*\sizeGridTree,0*\sizeGridTree) -- (0*\sizeGridTree, 2*\sizeGridTree);
\draw[color=black, line width=\LineWidth] (1*\sizeGridTree,0*\sizeGridTree) -- (0*\sizeGridTree,2*\sizeGridTree);
\draw[color=black, line width=\LineWidth] (3*\sizeGridTree,0*\sizeGridTree) -- (5*\sizeGridTree,2*\sizeGridTree);
\draw[color=black, line width=\LineWidth] (0*\sizeGridTree,2*\sizeGridTree) -- (0*\sizeGridTree,4*\sizeGridTree);
\draw[color=black, line width=\LineWidth] (3*\sizeGridTree,2*\sizeGridTree) -- (5*\sizeGridTree,4*\sizeGridTree);
\draw[color=black, line width=\LineWidth] (9*\sizeGridTree,2*\sizeGridTree) -- (7*\sizeGridTree,4*\sizeGridTree);
\draw[color=black, line width=\LineWidth] (11*\sizeGridTree,2*\sizeGridTree) -- (7*\sizeGridTree,4*\sizeGridTree);
\draw[color=black, line width=\LineWidth] (13*\sizeGridTree,2*\sizeGridTree) -- (14*\sizeGridTree,4*\sizeGridTree);
\draw[color=black, line width=\LineWidth] (15*\sizeGridTree,2*\sizeGridTree) -- (14*\sizeGridTree,4*\sizeGridTree);
\draw[color=black, line width=\LineWidth] (0*\sizeGridTree,4*\sizeGridTree) -- (5*\sizeGridTree,6*\sizeGridTree);
\draw[color=black, line width=\LineWidth] (3*\sizeGridTree,4*\sizeGridTree) -- (5*\sizeGridTree,6*\sizeGridTree);
\draw[color=black, line width=\LineWidth] (11*\sizeGridTree,4*\sizeGridTree) -- (7*\sizeGridTree,6*\sizeGridTree);
\draw[color=black, line width=\LineWidth] (14*\sizeGridTree,4*\sizeGridTree) -- (7*\sizeGridTree,6*\sizeGridTree);

\foreach \x/\y in { -1/0, 1/0, 3/0, 3/2, 5/2, 9/2, 11/2, 13/2, 15/2, 0/4, 3/4, 5/4, 11/4, 14/4, 5/6, -3/4 }{
    \filldraw[color=black] (\x*\sizeGridTree, \y*\sizeGridTree) circle (\SizePoint);
}

\foreach \x/\y in { 5/2, 5/4, 5/6 }{
    \filldraw[color=red] (\x*\sizeGridTree, \y*\sizeGridTree) circle (\SizePoint);
}
\foreach \x/\y in { 7/2, 7/4, 7/6 }{
    \filldraw[color=red] (\x*\sizeGridTree, \y*\sizeGridTree) circle (\SizePoint);
}
\end{tikzpicture}
\end{array}
$
\caption{Split a tree with height $k$ in $2(k\!-\!1)$ trees.\label{fig:tree_height_k}}
\end{figure}

Now, using the previous argument, we provide a different proof for the result of Proposition \ref{prop:two_way_fk} (already obtained in $\cite{BruKnuRic}$).
\begin{proposition}
\label{prop:two_way_fk}
Let $F_k$ be the Fibonacci polynomials. We have:
$$
F_k^2 - F_{k+1}F_{k-1} = z^{k-1}.
$$
\end{proposition}

\begin{proof}
From the previous result we know that $\mathcal{T}_{=k}=\mathcal{T}_{\leq k}-\mathcal{T}_{\leq k-1}=\frac{z^k}{F_{k}F{k+1}}$, whence $\frac{zF_{k}}{F_{k+1}} - \frac{zF_{k-1}}{F_{k}}=\frac{z^k}{F_{k}F_{k+1}}$. So we can conclude the statement.
\end{proof}

We recall that an \emph{ordered forest} (briefly a \emph{forest}) is a $t$-uple of ordered trees.
The size of a forest is the sum of the sizes of its trees. It is known that the forests of size $n$ are in bijection with the trees of size $n+1$: to uniquely obtain a forest from a tree, we just need to remove the root of the tree.

\section{A bijective construction for directed convex polyominoes}

We start presenting a map $\Phi$, which builds a pair of forests from a given parallelogram polyomino. We point out that the mapping $\Phi$ is borrowed from \cite{abbs14}. 

In a parallelogram polyomino $P$, let $V_P$ be the set of dots defined as follows:
\begin{itemize}
\item we enlighten $P$ from east to west and from north to south;
\item we put a dot in the enlightened cells of $P$ except for the rightmost cell (denoted by $E$) of the top row of $P$.
\end{itemize}

\noindent We define the pair of forests $\Phi(P)=(F_e,F_s)$ using the following rules:
\begin{itemize}
\item the set of nodes of $F_e$ and $F_s$ is $V_P$ and the roots of the trees of $F_e$ (resp. $F_s$) are the dots in the cells of the top row (resp. rightmost column) of $P$;
\item in a row (resp. column) of $P$, except for the top row and the rightmost column, the rightmost (resp. topmost) node is the 
father of each other node in the same row (resp. column), and these nodes are brothers ordered from east to west (resp. north to south);
\end{itemize}

Figure~\ref{phi} shows an example of $\Phi$. Let the size of a pair $(F_e,F_s)$ of forests be given by the sum of the sizes of $F_e$ and $F_s$.
\begin{proposition}
\label{prop:bijection_pp_pair}
The map $\Phi$ is a bijection between \pps\ with semi-perimeter $n$ and pairs of 
forests with size $n-2$.
\end{proposition}
\begin{proof}
By definition of $\Phi$, for each node $x$ we have:

\begin{itemize}
\item if $x$ is one of nodes of the top (resp. rightmost) row (resp. column) then $x$ is the root of a tree in $F_e$ (resp. $F_s$);
\item otherwise, by construction, $x$ has exactly one father, in fact: there is exactly one topmost node in the same column of $x$ or 
one rightmost node in the same row of $x$, but not both at the same time. 
\end{itemize}

Since a father always lies above or to the right of its sons, there
is no cycle in the graph obtained by means of $\Phi$
and the only nodes without fathers are the nodes of the top row and the rightmost column of $P$, $\Phi$ produces two ordered forests $F_e$ and $F_s$.
The size of $(F_e,F_s)$ is equal to the number of dots enlightened in the \pp\, which is exactly equal to the semi-perimeter of the \pp\ minus two.
Now we will prove that $\Phi$ is injective.
Let $P_1$ and $P_2$ be two \pps\ such that $P_1\neq P_2$. Since $P_1$ is different from $P_2$ then there exists a first step of the lower path (or the upper path) of $P_1$ different to the corresponding step in $P_2$ and so in one of the two forests of $\Phi(P_1)$ (or of $\Phi(P_2)$) we can find a father with more sons that the corresponding father in the corresponding tree of $\Phi(P_2)$ (or of $\Phi(P_1)$). We deduce that $\Phi(P_1) \neq \Phi(P_2)$.
Now we show surjectivity.
Since forests of size $n$ are bijective to trees of size $n+1$, then pairs of forests of size $n$ are bijective to pairs of trees with size $n+2$. Now, pairs of trees with size $n+2$ are in bijection with trees of size $n+2$ (see Fig.~\ref{phi2}). Since it is known that \pps\ and ordered trees are equinumerous \cite{Stanley}, we can conclude that $\Phi$ is a bijection.
\end{proof}

\begin{center}
\begin{figure}[htb]
$
\begin{array}{c}
\includegraphics[width=.35\linewidth]{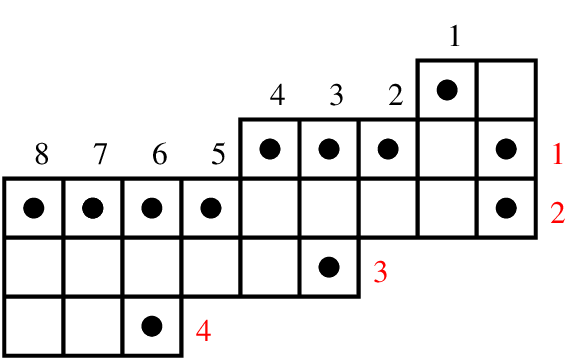}
\end{array}
\xleftrightarrow{\hspace{.2cm}}
\left(
\begin{array}{r}
F_e = 
\begin{array}{c}
\includegraphics[width=.02\linewidth]{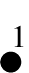}
\end{array}
,
F_s =
\begin{array}{c}
\includegraphics[width=.25\linewidth]{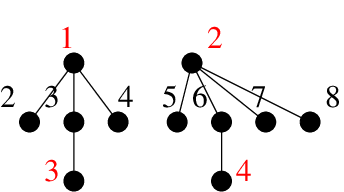}
\end{array}
\end{array}
\right)
$
\caption{The bijection $\Phi$.}\label{phi}
\end{figure}
\end{center}

\begin{figure}[htb]
\begin{center}
\includegraphics[width=.75\textwidth]{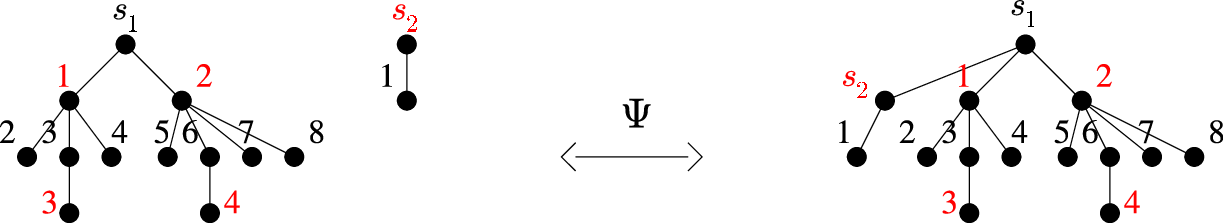}
\end{center}
\caption{The bijection $\Psi$ between parallelogram polyominoes and pairs of forests. \label{phi2}}
\end{figure}

Given a parallelogram polyomino $P$, we denote by $\alpha(P)$ (resp. $\beta(P)$) the number of cells in the top row (resp. rightmost column) of $P$ minus one.

Given a directed convex polyomino $D$ with a minimal bounding rectangle of size
$l_1 \times l_2$, we define $P_D$ to be the parallelogram polyomino obtained 
from $D$, extending the side of $D$ with ordinate equal to $l_2$ in the east 
direction to the point with coordinate $(l_1,l_2)$ and extending the side of 
$D$ with abscissa equal to $l_1$ in north direction to the point with
coordinate $(l_1,l_2)$(see Fig.~\ref{fig:PD}).

For instance, for the parallelogram polyomino $P_D$ in Fig.~\ref{fig:PD} we have $\alpha(P_D)=5$ and $\beta(P_D)=4$.

A monotone path starting with an east step and ending with a south step is said to be a \emph{cut of $P$} if 
the number of east steps is equal to $\alpha(P)+1$ and the number of south
steps is equal to $\beta(P)+1$.

We can easily check that in a directed convex polyomino $D$, the path $\lambda_D$ starting from the leftmost
corner of the top row (with an east step) and following clockwise the boundary of $D$ until it reaches the lowest corner of the rightmost column (with a south step) is a cut of $P_D$. 
For instance, with $D$ being the directed convex polyomino in Fig.~\ref{fig:PD}, $\lambda_D=e^2se^2se^2s^3$.

\begin{proposition}\label{prop:directed}
Directed convex polyominoes of semi-perimeter $n+2$ are in bijection with 
triplets $(F_e,F_s,\lambda)$ such that $F_e$ and $F_s$ are forests and
$\lambda$ is a monotone path, starting with an east step and ending with a south step, and having $m_e+1$ east steps and $m_s+1$ south steps. The integers $m_e$ and $m_s$ are respectively the numbers of trees in 
$F_e$ and $F_s$. 
The sum of the sizes of $F_e$ and of $F_s$ is equal to $n$.
\end{proposition}
\begin{proof}
A directed convex polyomino $D$ of semi-perimeter $n+2$ is uniquely determined by $P_D$ and $\lambda_D$. 
Now, using the bijection $\Phi$, the \pp\ $P_D$ is uniquely determined by a pair of 
forests $(F_e,F_s)$, with the property that the number of roots of 
$F_e$ (resp. $F_s$) is equal to $\alpha(P_D)$ (resp. $\beta(P_D)$).
By construction of $\phi$, the sum of the sizes of $F_e$ and of $F_s$ is equal to $n$.
Furthermore the cut has $\alpha(P_D) + 1$ east steps and $\beta(P_D) + 1$ 
south steps.

Conversely, we consider a triplet $(F_e, F_s, \lambda)$ satisfying the conditions
of Proposition~\ref{prop:directed}, we can build the \pp\ $\Phi^{-1}(F_e,F_s)$. We have that $\alpha(P)=m_e$ and $\beta(P)=m_s$. We deduce that $\lambda$ is a cut of $P$ and we can determine an unique directed convex polyomino.
\end{proof}

We recall that a \emph{bilateral Dyck path} is a directed path on $\mathbb{Z}\times \mathbb{Z}$ starting at $(0,0)$ in the $(x,y)$-plane and ending on the line $y=0$, which has unit steps in the $(1,1)$ (up step $u$) and $(1,-1)$ (down step $d$) directions. In the literature these objects are also referred to as free Dyck paths or Grand-Dyck paths.

A \emph{Dyck path} is a bilateral path which has no vertices with negative $y$-coordinates.

The \emph{semi-length} of a bilateral Dyck path is half the number of its steps. It is rather straightforward that the number of bilateral Dyck paths of semi-length $n$ is equal to ${2n} \choose {n}$. 
\begin{proposition}\label{prop:bilateral}
The number of bilateral Dyck paths of semi-length $n$ is equal to the number of 
triplets $(F_e,F_s,\lambda)$ where $F_e$ and $F_s$ are forests,
$\lambda$ is a monotone path, having $m_e$ east steps and $m_s$ south steps, and the integers $m_e$ and $m_s$ are respectively the number of trees in $F_e$ and $F_s$. 
The sum of the sizes of $F_e$ and of $F_s$ is equal to $n$.
\end{proposition}
\begin{proof}
Let $\chi=(F_e,F_s,\lambda)$ be a triplet satisfying the hypotheses.
We denote by $e_i$ (resp. $s_i$) the $i$-th east (resp. south) step of $\lambda$. We remark that the number of east (resp. south) steps of $\lambda$ is equal to the number of trees in $F_e$ (resp. $F_s$). We also recall that a tree $T$ can be represented by a Dyck path $D'(T)$, applying a standard mapping: we turn around the tree and we put an up step $u$ the first time we follow an edge of $T$ and a down step $d$ the second time. Then we build the elevated Dyck path $D(T)=uD'(T)d$. The semi-length of $D(T)$ is equal to the size of $T$. 

Now we associate to the $i$-th tree $T_i$ of $F_e$ the Dyck path $D(T_i)$, while to the $j$-th tree $T_j$ of $F_s$ the path $\overline{D}(T_j)$, obtained from $D(T_j)$ by exchanging $u$ steps with $d$ steps and viceversa. Now we define the bilateral Dyck path $B(\chi)$ as the concatenation of the paths $B_1\ldots B_{m_e + m_s}$, where:
\begin{itemize}
	\item $B_j=D(T_i)$ if the $j$-th step of $\lambda$ is $e_i$;
	\item $B_j=\overline{D}(T_i)$ if the $j$-th step of $\lambda$ is $s_i$,
\end{itemize}
where $j$ runs from $1$ to $m_e+m_s$.
We can proceed in the inverse way to obtain the triplet $(F_e,F_s,\lambda)$ starting from a given bilateral Dyck path $B$. We can see an instance of this bijection in Fig. \ref{fig:bil}.
\end{proof}

\begin{center}
\begin{figure}[htb]
$$
\left(
\begin{array}{r}
F_e = 
\begin{array}{c}
\includegraphics[width=.15\linewidth]{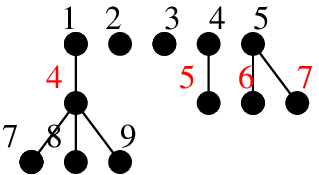}
\end{array}
,
F_s =
\begin{array}{c}
\includegraphics[width=.075\linewidth]{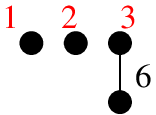}
\end{array}
, 
\lambda = 
e_1s_1e_2s_2e_3e_4e_5s_3
\end{array}
\right)\\
$$
$$
\updownarrow $$
$$
\begin{array}{r}
\includegraphics[width=.8\linewidth]{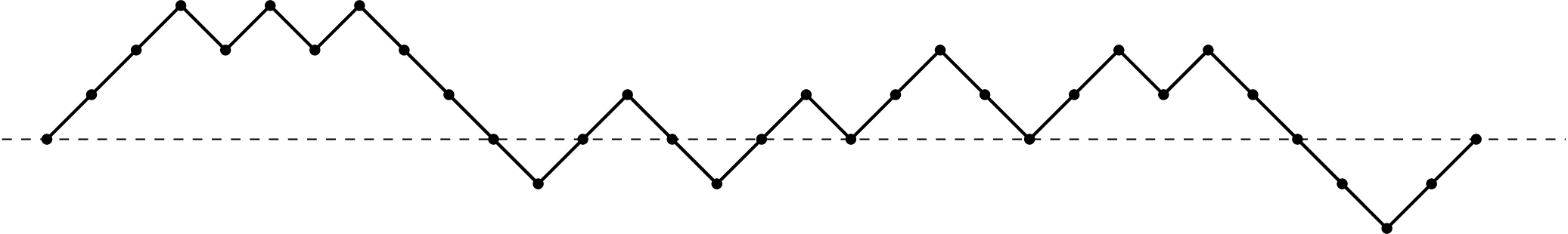}
\end{array}
$$
\caption{A triplet $(F_e,F_s,\lambda)$ and the corresponding bilateral Dyck path. \label{fig:bil}}
\end{figure}
\end{center}

\begin{corollary}
The number of bilateral Dyck paths of semi-length $n$ is equal to the number of directed convex polyominoes of semi-perimeter $n+2$.
\end{corollary}

\begin{proof}
Given a bilateral Dyck path $B$, let $(F_e,F_s,\lambda)$ be the triplet corresponding to $B$ according to the bijection described in Proposition \ref{prop:bilateral}. Now,  let us define $\lambda'=e\,\lambda\,s$. Finally, according to Proposition \ref{prop:directed} there is a unique directed convex polyomino $P$ corresponding to the triplet $(F_e,F_s,\lambda')$, then $P$ bijectively corresponds to $B$.
\end{proof}

Figure \ref{fig:bil} shows the bilateral Dyck path bijectively associated with the directed convex polyomino in Figure \ref{fig:Dir}.

\begin{center}
\begin{figure}[htb]
$
\begin{array}{c}
\includegraphics[width=.25\linewidth]{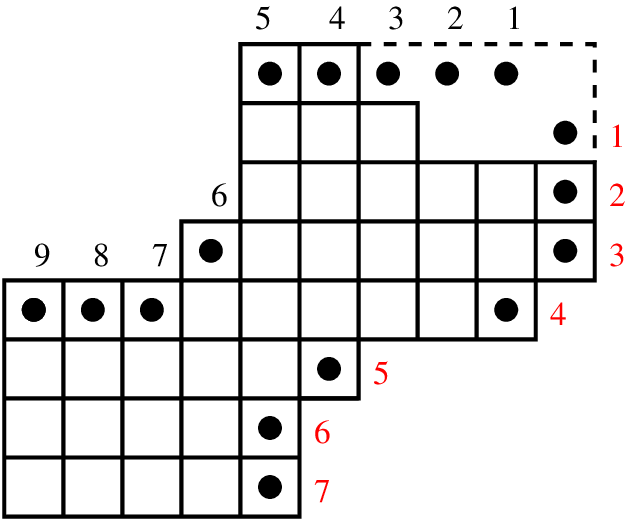}
\end{array}
\xleftrightarrow{\hspace{.2cm}}
\left(
\begin{array}{r}
F_e = 
\begin{array}{c}
\includegraphics[width=.15\linewidth]{Dir3}
\end{array}
,
F_s =
\begin{array}{c}
\includegraphics[width=.075\linewidth]{Dir2}
\end{array}
, 
\lambda = 
e^2sese^3s^2
\end{array}
\right)
$
\caption{
A directed convex polyomino with the corresponding triplet $(F_e,F_s,\lambda).$ \label{fig:Dir}
}
\end{figure}
\end{center}
As a consequence of Proposition \ref{prop:directed} we have a refinement of Proposition~\ref{prop:bijection_pp_pair}.

\begin{corollary}\label{cor:parallelogram}
A parallelogram polyomino of semi-perimeter $n+2$ is uniquely determined by a pair $(F_e,F_s)$ of ordered forests, where $F_e$ (resp. $F_s$) has exactly $\alpha(P)$ (resp. $\beta(P)$) ordered trees and the sum of the sizes of $F_e$ and of $F_s$ is equal to $n$.
\end{corollary}
\begin{proof}
We just need to observe that, a directed convex poly omino $D$ is a parallelogram 
polyomino if and only if $\lambda_D=e^{\alpha(P)+1}s^{\beta(P)+1}$.
\end{proof}

\section{A general method for the enumeration of directed convex polyominoes}
In this section,  we present a general method, based on Proposition~\ref{prop:directed},
to enumerate several statistics on families of directed convex polyominoes. Let $\cal C$ be a class of directed convex polyominoes. Our method consists of 3 steps:

\smallskip

{ \centering
\fbox{ \noindent \begin{minipage}{.97\linewidth}
\textbf{Step 1}:
Determine a combinatorial characterization of the cut $\lambda$ and the two forests 
$F_e$ and $F_s$ associated with the considered class $\cal C$ of directed convex polyominoes;

\textbf{Step 2:}
Determine the generating functions with respect to the studied statistics for all the 
cuts $\lambda$, and for all the trees of $F_e$ and $F_s$ associated with the class $\cal C$;

\textbf{Step 3:}
The generating function of the class $\cal C$ is then obtained by performing the composition of the generating function of the cuts with all the generating functions of the trees of $F_e$ and $F_s$.
\end{minipage}
}

}

\smallskip

This method is generic and very simple, and it allows us to obtain generating functions for several subclasses of directed convex polyominoes according to several different parameters, such as: the semi-perimeter, 
the degree of convexity, the width, the height, the size of the last 
row/column in $D$/$P_D$.

Let us begin with a classical example. We will count directed convex 
polyominoes according to the semi-perimeter.

\begin{proposition}
The generating function of directed convex polyominoes with respect to the semi-perimeter is given by
$$
\frac{z^2}{\sqrt{1-4z}}.
$$
\end{proposition}

\begin{proof} 
Let us apply the 3 steps of our method:

\noindent \textbf{Step 1:} 
According to Proposition~\ref{prop:directed}, directed convex polyominoes are in bijection with 
triplets $(F_e, F_s, \lambda)$ where: to each east step of $\lambda$ is bijectively associated a tree of $F_e$, except 
for the last east step of $\lambda$, and to each south step of $\lambda$ is associated a trees of $F_s$, except for the first 
south step. There is no constraint on the trees of $F_e$ and $F_s$.

\noindent \textbf{Step 2:} 
The generating function $\mathcal{T}$ for each ordered tree in $F_e$ and in 
$F_s$ is the following: 
$$
\mathcal{T} = \frac{1-\sqrt{1-4z}}{2}.
$$
The generating function for the cut is : 
$$
\mathcal{G}(s,e,z_s,z_e) = \frac{z_s\,z_e}{1-(s+e)},
$$
where $z_s$ (resp. $z_e$) represents the first (resp. last) south (resp. east) step of the cut.
while $e$ and $s$ represent all the other east and south steps of the cut.

\noindent \textbf{Step 3:} The generating function of directed convex polyominoes is
$$
\mathcal{G}(\mathcal{T},\mathcal{T},z,z) = \frac{z\,z}{1-2\mathcal{T}} = \frac{z^2}{\sqrt{1-4z}}.
$$ 

\end{proof}

\begin{proposition}
The generating function of directed convex polyominoes according to the semiperimeter,
the size of the top row and the size of the rightmost column of the associated parallelogram polyominoes is
$$
\frac{2xyz^2}{2-(x+y)(1-\sqrt{1-4.z})}.
$$
\end{proposition}
\begin{proof}
Let $z$, $x$, $y$ take into account the semi-perimeter, the size of the top row and the size of the rightmost column of the associated parallelogram polyominoes, respectively.
We just need to observe that the generating function for the cut is 
$\mathcal{G}(s, e, z_s, z_e ) = \frac{y\,z_s\, x\,z_e}{1-(y\,s+x\,e)}$.
So we obtain the desired generating function $\mathcal{G}(\mathcal{T},\mathcal{T},z,z)$.
\end{proof}

\begin{proposition}
The generating function of directed convex polyominoes according to the semiperimeter,
the size of the top row and the size of the rightmost column of the polyominoes is
$$
xy \, \frac{(1-\mathcal{T})^2}{(1-x\,\mathcal{T})(1-y\,\mathcal{T})} \, \frac{z^2}{\sqrt{1-4z}},
$$
where $\mathcal{T}=\frac{1-\sqrt{1-4z}}{2}$.
\end{proposition}
\begin{proof}
Let $z$, $x$, $y$ take into account the semi-perimeter, the size of the top row and the size of the rightmost column of the polyominoes, respectively. All south (resp. east) steps of the cut are labelled $s$ (resp. $e$), but, since there are not trees in bijection with the first south step and the last east step of the cut, we need to multiply the final generating 
function by $\frac{z_e\,z_s}{s\,e}$.

We need to take into account the first (resp. last) sequence of east (resp. south) steps of $\lambda$, which are labelled $x\,e$ (resp. $y\,s$), since they give the size of the top (resp. rightmost) row (resp. column) of the polyominoes. 
So the regular expression for the cut is the following:
$x\,e(x\,e)^*\left( 1 + s(e+s)^*e \right) (y\,s)^*y\,s$.
So the generating function is : 
$\mathcal{G}(s, e, z_s, z_e ) = \frac{z_e\,z_s}{e\,s} \left[ \frac{x\,e}{1-x\,e} \, \left( 1 + \frac{e\,s}{1-(e+s)} \right) \,\frac{y\,s}{1-y\,s} \right]$.
So we obtain the desired generating function $\mathcal{G}(\mathcal{T},\mathcal{T},z,z)$.
\end{proof}

\begin{proposition}[\cite{LinChang88, Bousquet1992}]
The generating function of directed convex polyominoes according to the semi-perimeter, the width and the height of the polyominoes is
$$
\frac{xyz^2}{\sqrt{ (1-(x+y)z)^2 -4xyz^2 } } \, .
$$
\end{proposition}
\begin{proof}
Let $z$, $x$, $y$ take into account the semi-perimeter, the width and the height of the polyominoes, respectively. 
Given a directed convex polyomino $D$, the width (resp. height) of $D$ is given by the number of cells of $P_D$ which are enlightened from north to south (resp. from east to west). 
By construction of $\Phi$, the cells contributing to the width (resp. height) of $D$ are 
associated with the nodes having an odd (resp. even) height in $F_e$ (resp. $F_s$)
and with the nodes having an even (resp. odd) height in $F_s$ (resp. $F_e$).
So we need the two generating functions $\mathcal{T}_e$ and 
$\mathcal{T}_s$ for the trees of $F_e$ and $F_s$, respectively. 
The nodes at odd (resp. even) height in $\mathcal{T}_e$ (resp. $\mathcal{T}_s$) 
are labelled by $x\,z$ and the nodes at even (resp. odd) height are labelled by $y\,z$.
The two generating functions are obtained by solving the following system:
$$
\mathcal{T}_e = \frac{x z}{1 - \mathcal{T}_s}\quad
\mbox{ and }
\quad \mathcal{T}_s = \frac{y z}{1 - \mathcal{T}_e} \, $$ from which we have
$$
\mathcal{T}_e(x,y) = \frac{1+(x-y)z-\sqrt{\left( 1+(x-y)z \right)^2-4xz} }{2}
\quad
\mbox{ and }
\quad
\mathcal{T}_s(x,y) = \mathcal{T}_e(y,x).$$
The generating function for the cut is 
$$\mathcal{G}(e, s, z_e, z_s ) = \frac{x z_e y z_s}{1-(e+s)},$$ where 
$z_s$ and $z_e$ represent the first south step and the last east 
step of the cut, while $e$ and $s$ represent all the other east and south steps of the cut, respectively.
So the final generating function is :
$$
\mathcal{G}(\mathcal{T}_e,\mathcal{T}_s,z,z)
=
\frac{xyz^2}{\sqrt{ (1-(x+y)z)^2 -4xyz^2 } }.
$$
\end{proof}

Now we consider a further statistic, i.e. the number of {\em inside/outside corners} of a directed convex polyomino. Let us recall that, following the boundary of a directed convex polyomino clockwise, any turn to the right (reps. left) is called an outside (resp. inside) corner. Also the corner at $(0,0)$ is an outside corner. Let us recall that the number of outside and inside corners of any polyomino differ in the amount of $4$. For instance, the polyomino in Fig.\ref{fig:Dir} has $11$ outside corners and $7$ inside corners.

\begin{proposition}\label{prop:corners}
The generating function of directed convex polyominoes according to the semi-perimeter and the number of outside corners is given by
$$
\frac{x^4 z^2}{1-2\mathcal{T}_c + (1-x)\mathcal{T}_c^2},
$$
where $\mathcal{T}_c = \frac{1 + (1-x)z - \sqrt{ (1 + (1-x)z)^2 - 4z }}{2}$.
\end{proposition}
\begin{proof}
The regular expression for the cut is $(ee^*xss^*)(ee^*xss^*)^*$ where $x$ counts
the number of east-south corners in the cut.
The generating function of the cut is equal to :
$$
\mathcal{G}(e,s,z_e,z_s)
=
\frac{z_e z_s}{es}
\cdot
\left[
	\frac{esx}{(1-e)(1-s)} 
	\cdot
	\frac{1}{1 - \frac{esx}{(1-e)(1-s)} }
\right]
=
\frac{x z_e z_s}{1 - (e+s) + (1-x)es}
.
$$

Except for the topmost and leftmost (resp. rightmost and lowest) outside corners of the directed convex polyomino $D$, the outside corners of type north-east and south-east are bijective with the first son of each node of the forests.
The generating function $T_c$ for the trees is obtained by solving the 
following equation:
$$
T_c = z( 1 + x( T_c+ T_c^2 + \dots ) ) = (1-x)z + \frac{zx}{1-T_c}.
$$
We obtain the generating function
$$
T_c = \frac{1 + (1-x)z - \sqrt{(1 + (1-x)z)^2 - 4z } }{2}.
$$

Now the desired generating function is: 
$$
x^3 \mathcal{G}( T_c, T_c, z, z ).
$$
In the formula above we have added the factor $x^3$, which takes into account the following three corners  of the polyomino: the topmost and leftmost outside corner; the rightmost and lowest outside corner; and the lowest and leftmost outside corner.
\end{proof}

\begin{corollary}\label{prop:insidecorners}
The generating function of directed convex polyominoes according to the semi-perimeter and the number of inside corners of polyominoes is given by
$$
\frac{z^2}{1-2\mathcal{T}_c + (1-x)\mathcal{T}_c^2},
$$
where $\mathcal{T}_c = \frac{1 + (1-x)z - \sqrt{ (1 + (1-x)z)^2 - 4z }}{2}$.
\end{corollary}
\begin{proof}
We have already observed that the number of outside corners of a polyomino is given by the number of inside corners plus $4$. 
So we obtain the desired generating function from the generating function for directed convex polyominoes with respect to the number of outside corners and semi-perimeter (obtained in Proposition \ref{prop:corners}) and multiplying by the factor $x^{-4}$.
\end{proof}

 We recall that the \emph{site-perimeter} of a polyomino is the number of nearest-neighbour vacant cells. This parameter is of considerable interest to physicists and probabilists.
In \cite{DG} the generating function for the parallelogram polyominoes with respect to the site-perimeter was computed and in \cite{MBR} the authors find the generating function with respect to the site-perimeter for the family of bargraphs polyominoes.

The enumeration according to corners lets us freely obtain the enumeration of directed convex polyominoes according to the site-perimeter. 

\begin{proposition}
The generating function of directed convex polyominoes according to the site-perimeter and the semi-perimeter is given by
$$
\frac{
	(y^2z)^2
}{
	1-2\mathcal{T}_s + 
	( 1 - y^{-1} ) \mathcal{T}_s^2
},
$$
where  $\mathcal{T}_s =  \mathcal{T}_c( y^{-1}, y^2z ) $ and
$\mathcal{T}_c(x,z) = \frac{1 + (1-x)z - \sqrt{ (1 + (1-x)z)^2 - 4z }}{2}$.

\end{proposition}
\begin{proof}
In a convex polyomino the site-perimeter is equal to the perimeter minus the number of inside corners.
According to Corollary \ref{prop:insidecorners}, let ${\cal D}_c(x,z)$ be the generating function for directed convex polyominoes where $x$ (resp. $z$) takes into account the number of inside corners (resp. the semi-perimeter) of polyominoes.
So, the generating function of directed convex polyominoes with respect to their site-perimeter and semi-perimeter is:
$${\cal D}_c(y^{-1},y^2z)$$
where $y$ (resp. $z$) takes into account the site-perimeter (resp. the semi-perimeter). 
\end{proof}

\begin{proposition}[\cite{Deutsch2003}]
The generating function of the symmetric directed convex polyominoes according 
to the semi-perimeter is given by
$$
\frac{z^2}{\sqrt{1 - 4z^2}}.
$$
\end{proposition}
\begin{proof}
Let $D$ be a symmetric directed convex polyomino and 
$\Phi(D) = (F_e, F_s, \lambda)$.
By construction of $\Phi$, since $D$ is symmetric,
$F_e = F_s$ and $\lambda$ is the concatenation of the half-cut $\lambda_{1/2}$
which is a path starting with an east step $e$, and the mirror of $\lambda_{1/2}$.

To compute the generating function, we need to determine the generating 
function of the half-cut and plug a tree inside. Then we need to count 
each node twice.

The generating function for $\lambda_{1/2}$ is
$$
\mathcal{G}(e,z_e) = \frac{z_e}{e} \left[ \frac{e}{1-(e+s)} \right].
$$

The generating function for the trees is $\mathcal{T}$.

The final result is 
$$
\mathcal{G}( \mathcal{T}(z^2), z^2 ) = \frac{z^2}{\sqrt{1 - 4z^2}}.
$$
\end{proof}

\section{Enumeration of directed $k$-convex polyominoes}
In this section we apply the method described in the previous section in order to obtain the generating function of directed $k$-convex polyominoes. We start recalling some basic definitions and enumerative results about $k$-convex polyominoes.  
\subsection{$k$-convexity constraint}
A \emph{path connecting two cells}, $b$ and $c$, of a convex polyomino $P$, is a path,
entirely contained in $P$, which starts from the center of $b$, and ends at 
the center of $c$ (see Fig.~\ref{paths}$(a)$).
Figure~\ref{paths}$(b)$ shows a monotone path connecting two cells of a polyomino.
Let $p=p_1\dots p_j$ be a path, each pair of 
steps $p_i p_{i+1}$ such that $p_i\neq p_{i+1}$, $1 \le i < j$, is called a 
\emph{change of direction} (see Fig.~\ref{paths}).

In \cite{Castiglione2003290} the authors prove that a polyomino is convex if and only if every pair of its cells can be joined by a monotone path.
Given $k\in\mathbb{N}$, a convex polyomino is said to be \emph{$k$-convex} if every pair of its cells can be connected by a monotone path with at most $k$ changes of direction. 
For the sake of clarity, we point out that a $k$-convex polyomino is also $h$-convex for every $h > k$.
We define the \emph{degree of convexity} of a convex polyomino $P$ as the smallest $k\in \mathbb{N}$ such that $P$ is $k$-convex.

Fig.~\ref{paths}$(b)$ shows a convex-polyomino with degree of convexity $k=2$.

\begin{figure}[htb]
\begin{center}
\includegraphics[width=.4\textwidth]{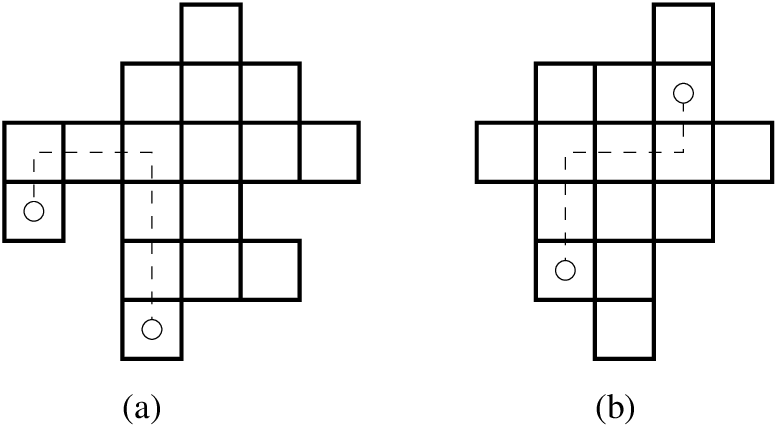}
\end{center}
\caption{$(a)$ a path connecting two cells of the polyomino with two changes of direction; $(b)$ a monotone path between two cells of the polyomino with two changes of direction. \label{paths}}
\end{figure}

In this paper we will deal with the class $\mathbb{D}_{\le k}$ (resp. $\mathbb{P}_{\le k}$) of directed $k$-convex polyominoes (resp. $k$-parallelogram polyominoes), i.e. the subclass of $k$-convex polyominoes which are also directed convex polyominoes (resp. parallelogram polyominoes).

Let $b=(i,j)$, $c=(i',j')$ be two cells of $D$. Without loss of generality, we can suppose that $i \le i'$. Now if $j'<j$, since $D$ is a directed convex polyomino, we can always join $b$ and $c$ by means of a monotone path with at most one change of direction (see Fig.~\ref{u-v}). 
Therefore, from now on, we will consider the case where $j'\geq j$. Let us define the \emph{bounce paths joining $b$ to $c$} as the two monotone paths internal to $D$ starting at $(i -\frac{3}{2},j-\frac{1}{2})$ (resp. at $(i-\frac{1}{2},j-\frac{3}{2})$) with an east (resp. north) unit step and ending at the center of $c$, denoted by $r_{b,c}$ (resp. $u_{b,c}$), where every side has maximal length (see Fig.~\ref{bouncepaths}$(a)$). 

\begin{figure}[htb]
\begin{center}
\includegraphics[width=.4\textwidth]{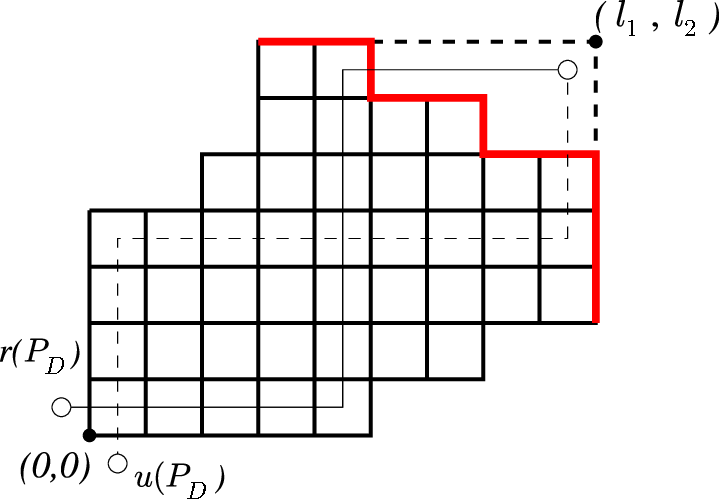}
\end{center}
\caption{$(a)$ A directed convex polyomino $D$ and its associated parallelogram polyomino $P_D$.}\label{fig:PD}
\end{figure}

We need to observe that the bounce paths joining two cells $b$ and $c$ are slightly different from the paths from $b$ to $c$, due to the presence of the tails. Observe that this is just a precaution to ensure that the two bounce paths are always different, and it will not affect the computation of the degree of convexity.   

The \emph{minimal bounce path joining $b$ to $c$} (denoted by $m_{b,c}$) is the bounce path joining $b$ to $c$ with the minimal number of changes of direction. If the two bounce paths joining $b$ to $c$ have the same number of changes of direction, by convention we define the minimal bounce path joining $b$ to $c$ to be the bounce path $r_{b,c}$.

In what follows, we will prove some combinatorial properties concerning directed convex polyominoes, understanding that they hold for parallelogram polyominoes a fortiori. 

\begin{proposition}
Let $D$ be a directed convex polyomino. Given $k\geq 1$, then $D$ is a directed $k$-convex polyomino if and only if $S$ can be connected to each cell of $D$ by means of a path having at most $k$ changes of direction. 
\end{proposition} 
\begin{proof}
($\Rightarrow$) It follows from the definition of directed $k$-convex polyomino.\\
($\Leftarrow$) We suppose, by contradiction, that $D$ is not a directed $k$-convex polyomino, then there exist two cells $a=(i_1,j_1)$ and $b=(i_2,j_2)$ (with $a,b\neq S$) which have to be connected by means of a path with at least $k+1$ changes of direction. We can suppose that $j_2> j_1$ and $i_2>i_1$ (otherwise we can find a path connecting the two cells with at most one change of direction). We consider the bounce path $m_{S,b}$: by hypothesis it has at most $k$ changes of direction; so it is possible to find a monotone path connecting $a$ to $b$ with at most $k$ changes of direction. To find this path, we just have to consider $m_{S,b}$ and, when it passes through a cell $c$ in the same column or row of $a$, we join $c$ to $a$. This is contradiction, so we have the proof.
\end{proof}

\begin{figure}[htb]
\begin{center}
\includegraphics[width=0.32\textwidth]{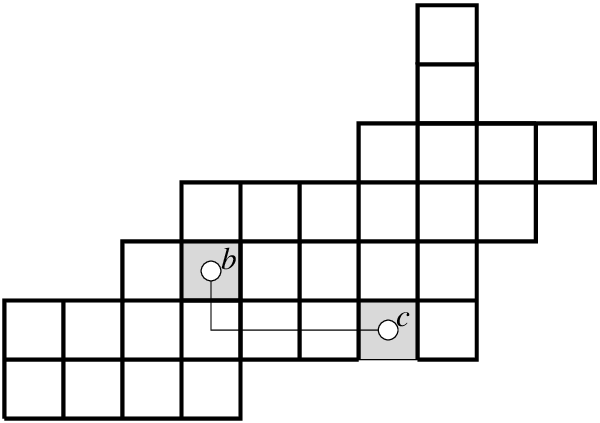}
\end{center}
\caption{Two cells $b=(i,j)$, $c=(i',j')$ of a directed convex polyomino where $j'<j$. \label{u-v}}
\end{figure}

\begin{lemma}
\label{lem:bounce}
Let $D$ be a directed convex polyomino and let $b$ and $c$ be two cells of $D$. The number of changes of direction of the minimal bounce path joining $b$ to $c$ is less than or equal to the number of changes of direction of any path joining $b$ to $c$. 
\end{lemma}
\begin{proof}
Let $p$ be a path joining $b=(i,j)$ to $c=(i',j')$ and starting at the centre of $b$ with an  east unit step.
The other case, where the starting step is a north step, is completely analogous.
We consider the change of direction of $p$ from east to north (similarly from north to east) occurring in the first cell $t$ such that there exists a cell $t'$ on the east (resp. north) of $t$. We remark that if such a cell $t'$ does not exist then, adding an unit east step to $p$ such that it starts at $(i -\frac{3}{2},j-\frac{1}{2})$  we obtain the bounce path $r_{b,c}$, hence the result follows.
Starting from $p$, we can build a new path $p'$, which has a change of direction at $t'$ instead of at $t$ and joins $b$ to $c$ with a number of changes of direction which is less then or equal to the number of changes of direction of $p$.
If we repeat this operation, and at the end we add an east step to the obtained path, such that it starts at $(i -\frac{3}{2},j-\frac{1}{2})$, we obtain the bounce path $r_{b,c}$ joining $b$ to $c$.
We can conclude that the number of changes of direction of the bounce path $r_{b,c}$ (resp. $u_{b,c}$) is less than or equal to the number of changes of direction of any path joining $b$ to $c$, starting with an east (resp. north) unit step.
So we have the result for the minimal bounce path joining $b$ to $c$.  
\end{proof}

Therefore, it is worth defining the \emph{degree of a cell $b$} of a directed convex polyomino as the number of changes of direction of the minimal bounce path joining $S$ to $b$. 

Given a \pp\ $P$, we denote by $E$ the rightmost cell of the top row of $P$.
We define the \emph{bounce paths} of $P$ to be the two bounce paths joining $S$ to $E$ and we denote by $r(P)$ (resp. $u(P)$) the bounce path $r_{S,E}$ (resp. $u_{S,E}$)(see Fig.~\ref{bouncepaths}$(b),(c)$).
Henceforth, if no ambiguity occurs, we will write $u$ (resp. $r$) in place of $u(P)$ (resp. $r(P)$).
The \emph{minimal bounce path of $P$} (denoted by $m(P)$ or $m$) is the minimal bounce path joining $S$ to $E$.
%
\begin{figure}[htb]
\begin{center}
\includegraphics[width=1\textwidth]{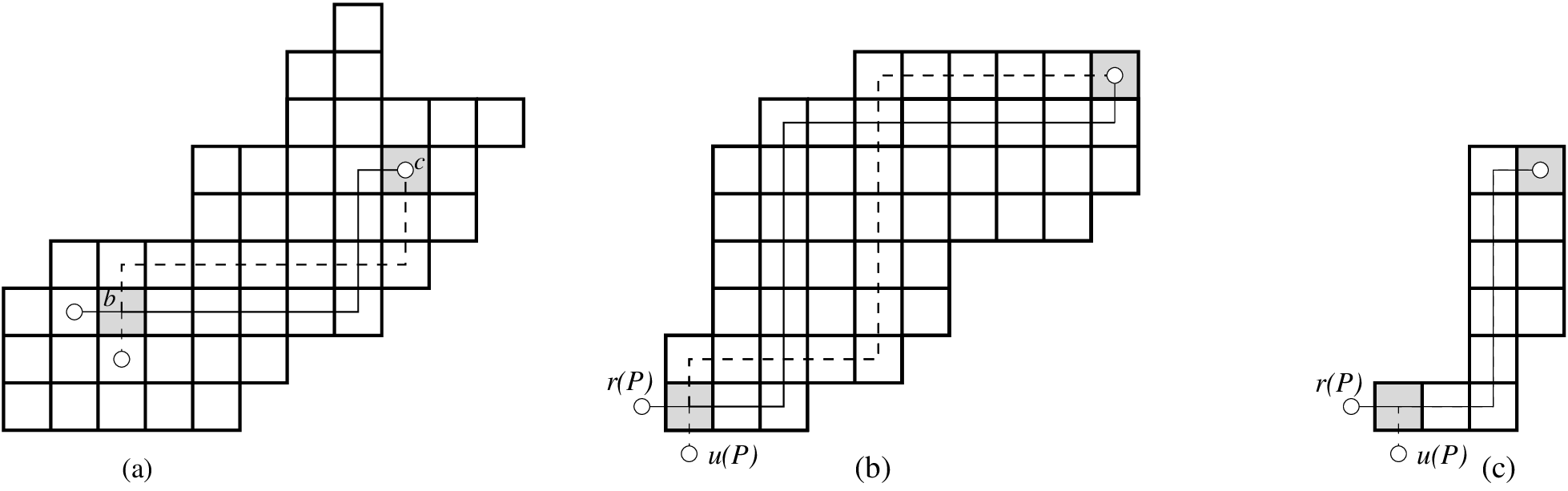}
\end{center}
\caption{$(a)$ The bounce paths joining two cells of a directed convex polyomino; $(b),(c)$ The bounce paths of a parallelogram polyomino. \label{bouncepaths}}
\end{figure}

\subsection{$k$-parallelogram polyominoes}

Now, we restrict our investigation to the case of $k$-parallelogram polyominoes. We will show that the convexity degree of a parallelogram polyomino can be expressed in terms of the heights of the two forests associated with the polyomino, via the mapping $\Phi$.

\medskip

Let $b$ and $c$ be two cells of a given \pp\ $P$. There may exist a cell starting 
from which the two paths $r_{b,c}$ and $u_{b,c}$ are superimposed. 
In this case, we denote such a cell by $J_{b,c}$ ($J$ is for ``joining").
Clearly $J_{b,c}$ may even coincide with $b$ and if such cell does not exist, 
we assume that $J_{b,c}$ coincides with $c$.
We denote by $J$ or $J(P)$ the cell $J_{S,E}$.

\begin{lemma}
\label{lem:difference_u_r}
Let $P$ be \pp\ and let $b$,$c$ be two cells of $P$. The number of changes of direction between $r_{b,c}$ and $u_{b,c}$
differ in the amount of $\pm 1$ if $J_{b,c} \neq c$ and by $0$ if $J_{b,c} = c$.
\end{lemma}

\begin{proof}
In the subpath between $b$ and $J_{b,c}$, each time $r_{b,c}$ crosses $u_{b,c}$, the number of 
changes of direction of $r_{b,c}$ and of $u_{b,c}$ increases by $1$.
If $J_{b,c}=c$ then the number of changes of direction of the two paths differ in the amount of $0$.
If $J_{b,c}\neq c$ then the number of changes of direction of the two paths differ in the amount of 
$\pm 1$. Precisely, the difference is equal to $1$ if $r_{b,c}$ has a change of direction at $J_{b,c}$ and to $-1$ 
if $u_{b,c}$ has a change of direction at $J_{b,c}$.
\end{proof}

\begin{lemma}
\label{lem:max_bounce}
Let $P$ be a \pp\ and let $b$ and $c$ be two cells of $P$. 
The number of changes of direction of the minimal bounce path $m$ is greater than or equal to the number of changes of direction of the minimal bounce path $m_{b,c}$ joining $b$ to $c$. 
\end{lemma}

\begin{proof}
Let $b$ and $c$ be two cells of a \pp\ $P$.
Let $u_{b,c}$ be the bounce path joining $b$ to $c$ (the argument is similar if we consider the bounce path $r_{b,c}$).
Now we consider the bounce path $u_{b,E}$ joining $b$ to $E$.
We observe that $u_{b,E}$ coincides with $u_{b,c}$ in the subpath from $b$ to the last change of direction of $u_{b,c}$.
We deduce that the number of changes of direction of $u_{b,E}$ is greater than or equal to the number of changes of direction of $u_{b,c}$.
By a similar argument, the number of changes of direction of $r_{b,E}$ is greater than or equal to the number of changes of direction of $r_{b,c}$.
We deduce that the number $h''$ of changes of direction of $m_{b,E}$ is greater than or equal to the number $h'$ of changes of direction of $m_{b,c}$.
Now we suppose by contradiction that the number $h$ of changes of direction of the minimal bounce path $m$ of $P$ is less than $h'$ (namely, the number of changes of direction of $m_{b,c}$). Then we can find a path $p$ joining $b$ to $E$, (in fact we consider $m=m_{S,E}$ and when it pass through a cell $t$ in the same column or row of $b$, we join $t$ to $b$) with a number of changes of direction less than or equal to $h$, while we supposed by contradiction that $h<h'$. We also observed above that $h'\leq h''$, and this contradicts Lemma~\ref{lem:bounce}, because we have a path $p$ joining $b$ to $E$ with number of changes of direction less than the number of changes of direction of $m_{b,E}$.  
\end{proof}

\begin{lemma}
The degree of convexity of a \pp\ is equal to the number of changes of direction of the minimal bounce path $m$ of the \pp.\
\label{lem:minimal_convexity}
\end{lemma}
\begin{proof}
Let $P$ be a \pp. 
According to Lemma~\ref{lem:bounce}, we just have to consider all the bounce paths of $P$.
Lemma~\ref{lem:max_bounce} states that the two cells $S$ and $E$ the two cells which require the maximal number of changes of direction to be connected.
We can conclude that the degree of convexity of $P$ is exactly equal to the number of changes of direction of the minimal bounce path $m$ of $P$.
\end{proof}

Let us now take in consideration the mapping $\Phi$, described in Proposition \ref{prop:bijection_pp_pair}, between \pps\ with semi-perimeter $n$ and pairs of forests with size $n-2$.
Let $r$ and $u$ be the two bounce paths of a \pp\ $P$.
At every change of direction the path $r$ (resp. $u$) individuates an enlightened cell. The sequence of these cells determines a path $d_r$ (resp. $d_u$) in a tree of $F_e$ or $F_s$ (see Figure~\ref{fig:dr_du} and Figure~\ref{fig:dr_du1} ). We observe that the last node of this path is a root of $F_e$ or $F_s$ and that its length is not necessarily maximal.   
By $l(d_r)$ (resp. $l(d_u)$) we denote the length of $d_r$ (resp. $d_u$) and we note that $l(d_r)$ (resp. $l(d_u)$) is the number of changes of direction of $r$ (resp. $u$).

\begin{center}
\begin{figure}[htb]
$
\begin{array}{c}
\includegraphics[width=.35\linewidth]{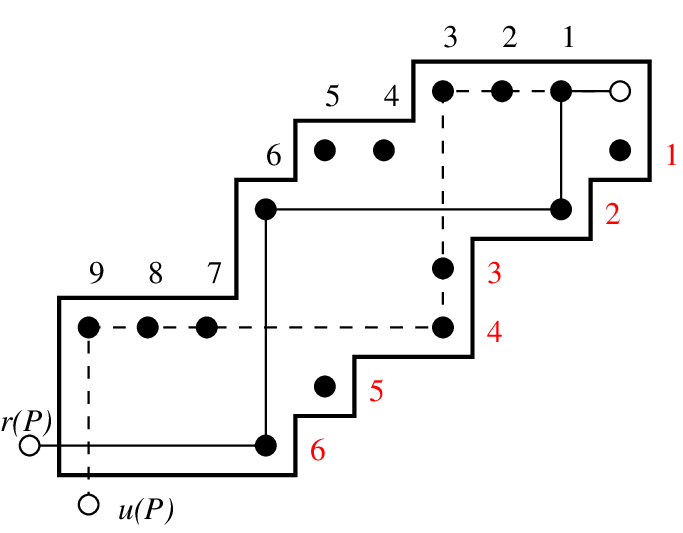}
\end{array}
\xleftrightarrow{\hspace{.2cm}}
\left(
\begin{array}{r}
F_e = 
\begin{array}{c}
\includegraphics[width=.15\linewidth]{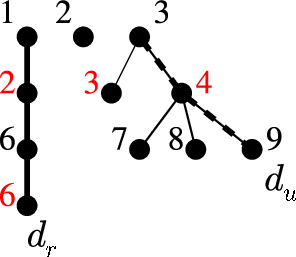}
\end{array}
,
F_s =
\begin{array}{c}
\includegraphics[width=.075\linewidth]{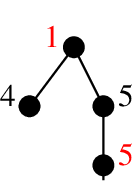}
\end{array}
\end{array}
\right)
$
\caption{The simple paths $d_r$ and $d_u$ associated with the bounce paths $r(P)$ and $u(P)$ respectively.\label{fig:dr_du}}
\end{figure}
\end{center}

\begin{center}
\begin{figure}[htb]
$
\begin{array}{c}
\includegraphics[width=.35\linewidth]{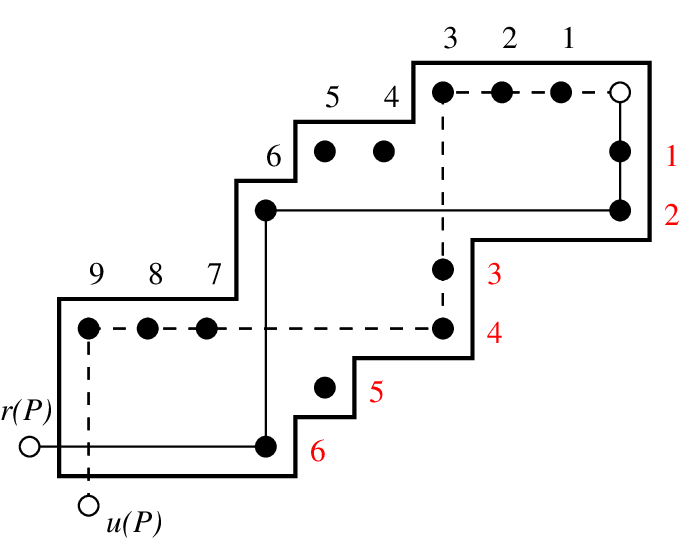}
\end{array}
\xleftrightarrow{\hspace{.2cm}}
\left(
\begin{array}{r}
F_e = 
\begin{array}{c}
\includegraphics[width=.15\linewidth]{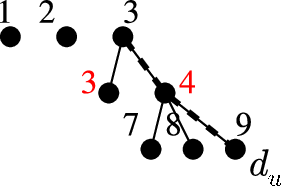}
\end{array}
,
F_s =
\begin{array}{c}
\includegraphics[width=.1\linewidth]{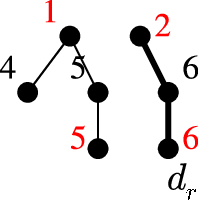}
\end{array}
\end{array}
\right)
$
\caption{The simple paths $d_r$ and $d_u$ associated with the bounce paths $r(P)$ and $u(P)$ respectively.\label{fig:dr_du1}}
\end{figure}
\end{center}


\begin{lemma}\label{lemma_h_d}
Let $P$ be a \pp\ and $\Phi(P)=(F_e,F_s)$. Then: 
$$ max(h(F_e), h(F_s)) = max( l(d_r), l(d_u)).$$
\end{lemma}
\begin{proof}
Let $p$ be a maximal path of a tree of $F_e$ (resp. $F_s$) starting at a root $r$ of $F_e$ (resp. $F_s$).
By construction of $\Phi$, we can find a monotone path $g$ internal to $P$, such that:
\begin{itemize}
\item $g$ is a bounce path joining the cells associated with the extremities of 
      $p$;
\item $g$ changes direction at each cell associated with the nodes of $p$, except for the root $r$.
\end{itemize}
Now we define $g'$ as the bounce path obtained from $g$ by joining $r$ to $E$. 
We remark that the length of $p$, namely the height of $F_e$ (resp. $F_s$), is equal to the number of changes of direction of $g'$. 

According to Lemma~\ref{lem:difference_u_r} and Lemma~\ref{lem:max_bounce}, we have that the number of changes of direction of $g'$ is less than or equal to the maximum between $l(d_r)$ and $l(d_u)$.
Hence we have: $\mbox{ max }(h(F_e),h(F_s))\leq \mbox{ max }(l(d_r),l(d_u)).$
Since $d_r$ and $d_u$ are two chains of $F_e$ or $F_s$, we conclude with the 
statement.
\end{proof}

\begin{lemma}\label{prop:equivalence_height}
Let $P$ be a \pp\ and $\Phi(P)=(F_e,F_s)$.
The following equivalences hold:
\begin{description}
	\item{i.} $h(F_e) = h(F_s) \Leftrightarrow d_r \text{ and } d_u \text{ are chains of different trees} \Leftrightarrow l(d_r) = l(d_u);$
	\item{ii.} $h(F_e) > h(F_s) \Leftrightarrow d_r \text{ and } d_u \text{ are chains of }F_e \Leftrightarrow |l(d_r)-l(d_u)|=1 \text{ and } d_r,d_u \in F_e;$
  \item{iii.} $h(F_s) > h(F_e) \Leftrightarrow d_r \text{ and } d_u \text{ are chains of }F_s \Leftrightarrow |l(d_r)-l(d_u)|=1 \text{ and } d_r,d_u \in F_s.$
\end{description}
\end{lemma}
\begin{proof}
We consider the two bounce paths $r$ and $u$ of $P$ and the cell $J(P)$.
There are two cases: $J=E$ or $J\neq E$.

\begin{description}
\item[case $J=E$.] The last change of direction of $r$ is in the top row or in the rightmost column.
Without loss of generality, we will suppose that the last change of direction of $r$ is 
in the top column.
Then, since $J=E$, the last change of direction of $u$ is in the rightmost column.
By construction of $\Phi$, we deduce that $d_r$ and $d_u$ are chains of trees of
different forests.

According to Lemma~\ref{lem:difference_u_r}, $l(d_r)=l(d_u)$.

To summarize, $d_r$ and $d_u$ are in trees of different forests, $l(d_r)=l(d_u)$ and 
Lemma~\ref{lemma_h_d} says that $max(h(F_e), h(F_s)) = max( l(d_r), l(d_u) )$.
It follows that $h(F_e) = h(F_s)$.

\item [case $J\neq E$.] 
According to Lemma~\ref{lem:difference_u_r}, $| l(d_r) - l(d_u) |=1$.

Without loss of generality, we will suppose that the number of changes of direction of $r$ is greater than the number of changes of direction of $u$. Starting from $J\neq E$ the two paths $r$ and $u$ are superimposed and so the chains $d_r$ 
and $d_u$ are in the same tree of $F_e$ or $F_s$.

Without loss of generality, we will suppose that $d_r$ and $d_u$ are chains of a tree of $F_e$. 
According to Lemma~\ref{lemma_h_d}, it follows that $h(F_e) \ge h(F_s)$, we just have to prove that $h(F_e) \neq h(F_s)$.
We suppose, that $h(F_e)=h(F_s)$.  
We consider the maximal chain $d_1$ (resp. $d_2$) of $F_e$ (resp. $F_s$) where its extremity is the rightmost leaf of $F_e$ (resp. $F_s$). By construction of $\Phi^{-1}$, we can prove that $d_r$ and $d_u$ coincide with $d_1$ and $d_2$.
And this leads to a contradiction because of $d_r$ and $d_u$ are chains of the same tree in $F_e$.
\end{description}
\end{proof}

\begin{proposition}
Let $P$ be a \pp, and $\Phi(P)=(F_e,F_s)$.
The degree of convexity of $P$ is equal to
$$
max( h(F_e), h(F_s)) -
\left\{
\begin{array}{ll}
0 & \text{if }  h(F_e) = h(F_s); \\
1 & \text{otherwise}.
\end{array}
\right.
$$
\label{prop:minimal_convexity_PP}
\end{proposition}

\begin{proof}
According to Lemma~\ref{lem:minimal_convexity}, we know that the degree of convexity of a \pp\ $P$ is equal to $min( l(d_r), l(d_u) )$.
According to Lemma~\ref{prop:equivalence_height},
\begin{itemize} 
\item if $h(F_e) = h(F_s)$, then $l(d_r) = l(d_u)$. 
Using Lemma~\ref{lemma_h_d}, the degree of convexity of $P$ is equal to $max(h(F_e), h(F_s))$ .
\item if $h(F_e) \neq h(F_s)$, then $|l(d_r) - l(d_u)| =1$.
Using Lemma~\ref{lemma_h_d}, the degree of convexity of $P$ is equal to $max(h(F_e), h(F_s))-1$.
\end{itemize}
\end{proof}

We are now ready to calculate the generating function for $k$-parallelogram polyominoes, which was determined for the first time in \cite{BatFedRinSoc}, by using a purely analytic method. The following proposition provides a bijective proof of this enumerative result.
\begin{proposition}[\cite{BatFedRinSoc}]\label{prop:fgkpar}
The generating function of $k$-\pps\ with respect to the semi-perimeter is given by
$$
\mathcal{P}_{\le k} = z^2\left( \frac{F_{k+2}}{F_{k+3}} \right)^2 - z^2
\left( \frac{F_{k+2}}{F_{k+3}} - \frac{F_{k+1}}{F_{k+2}} \right)^2,
$$
where $F_k$ are the Fibonacci polynomials.
\end{proposition}
\begin{proof}
First we recall that ordered trees of size $n+1$ are bijective to ordered forests with size $n$ by removing the root; so it follows that the function $f =\left( \frac{F_{k+2}}{F_{k+3}} \right)^2$ is the 
generating function of pairs of ordered forests with height less than or equal to $k+1$ with respect to their size.
According to Proposition~\ref{prop:minimal_convexity_PP}, to obtain $\mathcal{P}_{\leq k}$
we need to remove from $f$ the generating function 
$\left( \frac{F_{k+2}}{F_{k+3}} - \frac{F_{k+1}}{F_{k+2}} \right)^2$ of pairs of ordered forests with height exactly equal to $k+2$ and finally multiply by $z^2$.
\end{proof}

\subsection{Directed $k$-convex polyominoes}

In this section we extend the method used in the previous section in order to count directed $k$-convex polyominoes.
Let $D$ be a directed convex polyomino; a sequence of cells of $P_D$ is naturally determined by the points where the bounce 
paths $r(P_D)$ and $u(P_D)$ cross each other. The sequence of these cells (denoted by $L_h=(i_h,j_h)$) starts from $S$ and ends with $J$. Let $M$ be the index such that $L_M=J$.
We extend $(L_h)_h$ to a new sequence $(I_h)_h$ by defining 
$I_{h\le M} = L_h$,  
and $I_{h > M}$ by the cells where the two superimposed bounce paths 
change direction. Each cell $I_h$ will be labelled by $h$.

From now on, we will refer to the set of directed convex polyominoes such that $J(P_D)=E$, as {\em flat} directed convex polyominoes. 

\begin{lemma}\label{lem:degree_ih}
The degree of the cell $I_h$ is $h$ if $h \leq M$, otherwise is $h-1$.
\end{lemma}
\begin{proof}
By definition, the degree of $I_h$ is the number of changes of direction of the minimal 
bounce path $m_{S,I_h}$ joining $S$ to $I_h$. By construction of $I_h$, 
$m_{S,I_h}$ coincides with a first part of the minimal bounce path $m$ of the \pp\ $P_D$ associated with the directed convex polyomino $D$; so we can consider $m$ in place of $m_{S,I_h}$. The minimal bounce path $m$ changes direction to reach $I_{h+1}$ from $I_h$, and, since the degree of $I_0$ is $0$, then the degree of $I_h$ is $h$ if $h\leq M$. Furthermore, at $I_{M+1}$, the bounce path $m$ does not change direction and so the degree of $I_{M+1}$ is $M$. If $h>M+1$, the path $m$ changes direction at each $I_h$, hence, in this case, the degree of $I_h$ is $h-1$. 
\end{proof}

\begin{corollary}\label{lem:maximum_label}
Let $D$ be a directed convex polyomino, the maximal label of $I_h$ is equal to the degree of 
convexity of $P_D$.
\end{corollary}
\begin{proof}
If $D$ is flat polyomino, we have $I_{M} = J(P_D)=E$, so $M$ is the degree of convexity of $P_D$.
Otherwise, at $I_l\neq E$ the two superimposed bounce paths have a last change of direction, before reaching $E$. So, according to previous Proposition, the degree of $I_l$ is $l-1$, then the degree of $E$ is $l$. 
\end{proof}

\begin{corollary}\label{cor:maximum_label}
Let $D$ be a directed convex polyomino such that the degree of convexity of $P_D$ 
is $k$. The sequence $(I_h)_h$ has length $k$. If $D$ is flat 
then $I_k=E$; otherwise $I_k \not= E$ and $I_k$ is the cell where the minimal bounce path $m$ of $P_D$ has the last change of direction.
\end{corollary}

\begin{proposition}
Let $h,\eta$ be two integers and $\eta\geq 1$, we consider the cell $I_h=(i_h,j_h)$. If the cell $d=(i_h+1, j_h+\eta)$ (resp. $d=(i_h+\eta,j_h+1)$) exists, then the degree of $d$ is exactly $h+1$.
\end{proposition}
\begin{proof}
First we observe that the minimal bounce path $m_{S,d}$ always passes trough $I_h$.
Now, we have to consider two cases.
In the first case we suppose that $h\leq M$, then the degree of $I_h$ is exactly $h$, the two bounce paths do not change direction at $I_h$ and, since that to reach $d$ from $I_h$ we have to do one change of direction, it follows that the degree of $d$ is $h+1$.
Otherwise, we have $h>M$, then the degree of $I_h$ is $h-1$, the two superimposed bounce paths change direction at $I_h$ and, since that to reach $d$ from $I_h$ we have to do one change of direction, we conclude that the degree of $d$ is $h-1+2=h+1$.
\end{proof}

\begin{proposition}\label{prop:labels}
Let $D$ be a directed convex polyomino. Let us consider the cells $I_{h-1}=(i_{h-1},j_{h-1})$ and $I_h=(i_h,j_h)$.
Then we have:
\begin{enumerate}[1)]
	\item \label{label:case_1} each cell $(i,j)$ of $P_D$ (resp. $D$) with $i_{h-1} < i\leq i_h$ and $j_{h-1} < j$ has degree $h$;
	\item \label{label:case_2} each cell $(i,j)$ of $P_D$ (resp. $D$) with $j_{h-1} < j \leq j_h$ and $i_h < i$ has degree $h$;
	\item\label{label:case_3} each cell $(i,j)$ of $P_D$ (resp. $D$) with $i_h < i$ and $j_h < j$ has degree greater than $h$.
\end{enumerate}
\end{proposition}

The content of Proposition~\ref{prop:labels} is illustrated by Figure~\ref{fig:label}.

\begin{proof}
We remark that, if we consider two cells $c=(i,j)$ and $d=(k,l)$ such that $i<k$ and $j<l$, 
then the degree of $c$ is less than or equal to the degree of $d$.
Since that the degree of the cell $(i_h+1, j_h+\eta)$ (resp. $(i_h+\eta,j_h+1)$), where $\eta\geq 1$, is equal to $h+1$, then we have that case~\ref{label:case_3} holds for every $h$.

Now, we suppose that $h>M$, then at $I_h$ the bounce path changes direction and $I_h$ and $I_{h-1}$ are aligned.
Without loss of generality we can suppose that at $I_h$ the path $m$ has a change of direction of type east-north. So, the cells $I_h$ and $I_{h-1}$ are aligned horizontally, then case~\ref{label:case_2} holds because there are not cells in this area.
The cells on top of $I_h$ in the same column have the degree equal to $h$ as each cell $(i_{h-1}+1, j_{h-1}+\eta )$ with $\eta\geq 1$. So case~\ref{label:case_1} is verified. 

Now we suppose that $h\leq M$, then the cells $I_h$ and $I_{h-1}$ are not aligned.
The cells on top of $I_h$ (resp. $I_{h-1}$) in the same column and the cells 
on the right of $I_h$ (resp. $I_{h-1})$ in the same row have the degree equal to 
$h$ (resp. $h-1$).
Furthermore, each cell $(i_{h}+1, j_{h}+\eta)$ (resp. $(i_h+\eta,j_h+1)$) with $\eta\geq 1$ has the degree equal to $h+1$, while the cell $(i_{h-1}+1, j_{h-1}+1)$ has the degree equal to $h$.
We deduce that case \ref{label:case_1} and case \ref{label:case_2} hold.
\end{proof}


\begin{figure}[htb]
$$
\begin{array}{c}
\begin{tikzpicture}[scale=.4]                                                   
\AddSquare{0}{0}
\AddSquare{0}{1}
\AddSquare{0}{2}

\AddSquare{1}{0}
\AddSquare{1}{1}
\AddSquare{1}{2}

\AddSquare{2}{0}
\AddSquare{2}{1}
\AddSquare{2}{2}
\AddSquare{2}{3}
\AddSquare{2}{4}

\AddSquare{3}{1}
\AddSquare{3}{2}
\AddSquare{3}{3}
\AddSquare{3}{4}

\AddSquare{4}{2}
\AddSquare{4}{3}
\AddSquare{4}{4}

\AddSquare{5}{2}
\AddSquare{5}{3}
\AddSquare{5}{4}

\AddSquare{6}{3}
\AddSquare{6}{4}
\AddSquare{6}{5}

\AddSquare{7}{4}
\AddSquare{7}{5}
\AddSquare{7}{6}{4}

\AddSquare{8}{5}
\AddSquare{8}{6}

\AddSquare{9}{6}
\AddSquare{9}{7}

\AddSquare{10}{6}{5}
\AddSquare{10}{7}
\AddSquare{10}{8}

\node at \Point{12}{-0.5} {$h\!\!=\!\!0$};

\draw[color=red, line width=1.5*\LineWidth] (.5,.5) -- (.5, 10.5)  ;
\draw[color=red, line width=1.5*\LineWidth] (.5,.5) -- (12.5, .5)  ;

\node at \Point{12}{1.5} {$h\!\!=\!\!1$};

\draw[color=blue, line width=1.5*\LineWidth] (2.5,2.5) -- (2.5, 10.5)  ;
\draw[color=blue, line width=1.5*\LineWidth] (2.5,2.5) -- (12.5, 2.5)  ;

\node at \Point{12}{3.5} {$h\!\!=\!\!2$};

\draw[color=red, line width=1.5*\LineWidth] (5.5,4.5) -- (5.5, 10.5)  ;
\draw[color=red, line width=2.5*\LineWidth] (5.5,4.5) -- (12.5, 4.5)  ;

\draw[color=blue, line width=2.5*\LineWidth] (7.5,4.5) -- (7.5, 10.5)  ;
\draw[color=blue, line width=1.5*\LineWidth] (7.5,4.5) -- (12.5, 4.5)  ;

\node at \Point{6.5}{10} {$h\!\!=\!\!3$};
\node at \Point{12}{5.5} {$h\!\!=\!\!4$};

\draw[color=red, line width=1.5*\LineWidth] (7.5,6.5) -- (7.5, 10.5)  ;
\draw[color=red, line width=2.5*\LineWidth] (7.5,6.5) -- (12.5, 6.5)  ;

\node at \Point{9}{10} {$h\!\!=\!\!5$};

\draw[color=blue, line width=1.5*\LineWidth] (10.5,6.5) -- (10.5, 10.5)  ;
\draw[color=blue, line width=1.5*\LineWidth] (10.5,6.5) -- (12.5, 6.5)  ;

\node at \Point{12}{10} {$h\!\!>\!\!5$};

\Ih{0}{0}{0}{red}
\Ih{2}{2}{1}{blue}
\Ih{5}{4}{2}{red}
\Ih{7}{4}{3}{blue}
\Ih{7}{6}{4}{red}
\Ih{10}{6}{5}{blue}

\end{tikzpicture}
\hspace{1cm}
\end{array}
\begin{array}{c}
\begin{tikzpicture}[scale=.4]                                                   
\AddSquare{0}{0}
\AddSquare{0}{1}
\AddSquare{0}{2}

\AddSquare{1}{0}
\AddSquare{1}{1}
\AddSquare{1}{2}

\AddSquare{2}{0}
\AddSquare{2}{1}
\AddSquare{2}{2}
\AddSquare{2}{3}
\AddSquare{2}{4}

\AddSquare{3}{1}
\AddSquare{3}{2}
\AddSquare{3}{3}
\AddSquare{3}{4}

\AddSquare{4}{2}
\AddSquare{4}{3}
\AddSquare{4}{4}

\AddSquare{5}{2}
\AddSquare{5}{3}
\AddSquare{5}{4}

\node at \Point{7}{-0.5} {$h\!\!=\!\!0$};

\draw[color=red, line width=1.5*\LineWidth] (.5,.5) -- (.5, 6.5)  ;
\draw[color=red, line width=1.5*\LineWidth] (.5,.5) -- (8.5, .5)  ;

\node at \Point{7}{1.5} {$h\!\!=\!\!1$};

\draw[color=blue, line width=1.5*\LineWidth] (2.5,2.5) -- (2.5, 6.5)  ;
\draw[color=blue, line width=1.5*\LineWidth] (2.5,2.5) -- (8.5, 2.5)  ;

\node at \Point{7}{3.5} {$h\!\!=\!\!2$};

\draw[color=red, line width=1.5*\LineWidth] (5.5,4.5) -- (5.5, 6.5)  ;
\draw[color=red, line width=1.5*\LineWidth] (5.5,4.5) -- (8.5, 4.5)  ;

\Ih{0}{0}{0}{red}
\Ih{2}{2}{1}{blue}
\Ih{5}{4}{2}{red}

\node at \Point{7}{5.5} {$h\!\!>\!\!2$};

\end{tikzpicture}
\end{array}
$$
\caption{The degrees of the cells inside a flat and a non flat polyomino parallelogram polyomino. \label{fig:label}}
\end{figure}
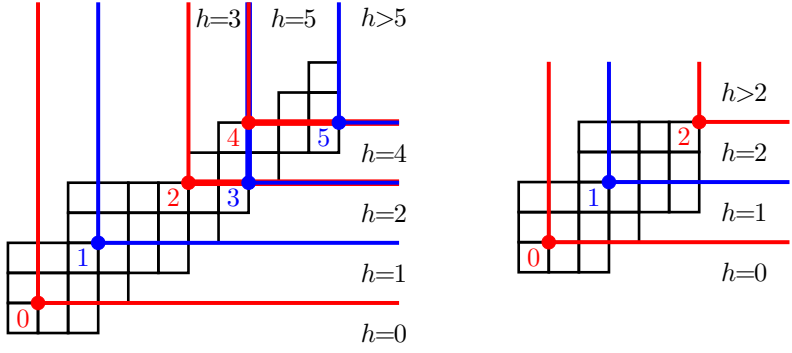

Now, given a parallelogram polyomino $P$ with degree of convexity equal to $k$, we define $R_P$ as the set of cells 
$c=(i,j)$ of $P$ such that $i_{k-1} < i$ and $j_{k-1} < j$.

\begin{corollary}\label{cor:k}
Let $D$ be a directed convex polyomino such that $P_D$ has degree of convexity 
$k$. 
Then the cells of $R_{P_D}$ (resp. $R_{P_D} \cap D$) are the only cells with 
degree $k$ in $P_D$ (resp. $D$).
\end{corollary}

\begin{proof}
According to Proposition~\ref{prop:labels}, the cells of $R_{P_D}$ (resp. $R_{P_D} \cap D$) have degree greater than
or equal to $k$.
Furthermore, the degree of convexity of $P_D$ is $k$, so we deduce that the cells of $R_{P_D}$ (resp. $R_{P_D} \cap D$) are the only cells with degree $k$ in $P_D$ (resp. $D$).
\end{proof}

\begin{lemma}
\label{lem:rectangle}
If $D$ is a flat directed convex polyomino, then $R_{P_D}$ is a non empty rectangle.
\end{lemma}
\begin{proof}
Les $D$ be a flat directed convex polyomino, $P_D$ the parallelogram polyomino associated with $D$ and $k$ the convexity degree of $P_D$. By construction, the cells $I_k=E$ and $(i_{k-1}+1,j_{k-1}+1)$ belong to $R(P_D)$.
Now we observe that the two bounce paths of $P_D$ have their last change of direction at two cells: one placed in the same column of $I_{k-1}$ and in the same row of $I_k=E$, the other one placed in the same row of $I_{k-1}$ and in the same column of $I_k=E$.
So, $R_{P_D}$ is a non empty rectangle.
\end{proof}

Let $P$ be a flat \pp\ with degree of  convexity $k$. Let $a$ and $b$ be the width and the height of $R_P$, respectively.
We denote by $\lambda_{R_P}$ the cut $ee^{\alpha(P)-a}s^be^as^{\beta(P)-b}s$ of $P$, which is precisely the cut of the directed convex polyomino obtained removing $R_P$ from $P$.

The following statement gives a characterization of the class $\mathbb{D}_{\le k}$, which will lead us to the desired generating function.

\begin{proposition}
\label{prop:charactkdir}
Every directed $k$-convex polyomino $D$ is uniquely determined by one of the 
two (mutually exclusive) situations: 
\begin{enumerate}[1)]
	\item \label{des:case_1} a $k$-parallelogram polyomino $P$ and a cut $\lambda$, which is a cut 
	of $P$, or,
	\item \label{des:case_2} a flat $k\!+\!1$-parallelogram polyomino $P$, with degree of convexity 
	$k+1$, and a cut $\lambda$, which is a cut of $P$ such that 
$\lambda \subseteq \lambda_{R_P}$, where the notation $\lambda \subseteq \lambda_{R_P}$ is used to mean that the path $\lambda_{R_P}$ is weakly above~$\lambda$.
\end{enumerate}
\end{proposition}

\begin{proof}
We recall that a directed convex polyomino $D$ is uniquely determined by 
$P_D$ and the cut $\lambda_D$ of $P_D$. So, we just need to prove that $D$ is a 
directed $k$-convex polyomino if and only if $P_D$ and $\lambda_D$ have the 
properties listed in case~\ref{des:case_1} or in case~\ref{des:case_2}.\\
$(\Rightarrow)$ We proceed by contradiction, so we have to consider three cases:
\begin{itemize}
	\item $P_D$ is a flat $k\!+\!1$-parallelogram polyomino, with degree of 
	convexity equal to $k+1$ and $\lambda_D\nsubseteq \lambda_{R_{P_D}}$. 
	Since $\lambda_D\nsubseteq \lambda_{R_{P_D}}$, then 
	$R_{P_D}\cap D\neq\varnothing$ and, according to Corollary~\ref{cor:k}, 
	$D$ is not a directed $k$-convex polyomino, against our hypothesis. 
	In Fig.~\ref{fig:case1} we can see a directed convex polyomino $D$ (which 
	is not a $2$-directed convex polyomino), where $P_D$ is a flat $3$-parallelogram 
	polyomino with degree of convexity equal to $3$, $R_{P_D}\neq\varnothing$, 
	and $\lambda_D=e^3sese^3s \nsubseteq \lambda_{R_{P_D}}=e^2s^2e^5s$. 
	\item $P_D$ is a $k+1$-parallelogram polyomino which is not flat, with 
	degree of convexity equal to $k+1$. 
	From Corollary~\ref{cor:maximum_label}, $I_{k+1}\neq E$ and $I_{k+1}$ is the cell where the minimal bounce path $m$ has the 
	last change of direction. 
	Without loss of generality, we can suppose that the change of direction at $I_{k+1}$ is of type north-east.
	The cells of the rightmost column of $D$ have the same degree as the cell $E$, 
	namely $k+1$.
	Furthermore, we recall that the last step of the cut $\lambda_D$ is a south step, so the lowest cell of the 
	rightmost column of $P_D$ belongs to $D$, it follows that $D$ is not a directed $k$-convex 
	polyomino, against our hypothesis.
	\item  $P_D$ is a $h$-parallelogram polyomino, with degree of convexity 
	equal to $h\geq k+2$. If $\lambda_D$ is a cut of $P_D$ such that 
	$\lambda_D\subseteq \lambda_{R_{P_D}}$, then $D$ will have degree of 
	convexity equal to $h-1$; in every other case $D$ will have degree of 
	convexity equal to $h$.
	In each case $D$ is not a directed $k$-convex polyomino, contradicting our 
	hypothesis.
\end{itemize}
$(\Leftarrow)$ If $P_D$ and $\lambda_D$ satisfy the properties of 
case~\ref{des:case_1} then trivially $D$ is a directed $k$-convex polyomino, 
while the statement in case~\ref{des:case_2} follows from 
Corollary~\ref{cor:k}.
\end{proof}

Our aim is to use the three steps of our method to obtain the generating function of directed $k$-convex polyominoes. However, since  the characterization of the cut turns out to be quite complex, it will be easier to obtain the desired generating function by difference.

Let $\mathbb{D}_{\le k}^-$ be the class of directed convex polyominoes which are bijective to triplets 
$(F_e, F_s, \lambda)$ where $\lambda$ is a generic cut, and $F_e$ and $F_s$ 
are forests with heights less than or equal to $k$, according to the mapping described in Proposition~\ref{prop:directed}. 
Proposition~\ref{prop:charactkdir} and the bound on the height of the forests ensure that $\mathbb{D}_{\le k}^-$ is included in $\mathbb{D}_{\le k}$.
Now $\mathbb{FD}_{=k}$ is the set of the flat directed convex 
polyominoes with degree of convexity $k$.
The following result can be proved using 
Proposition~\ref{prop:charactkdir} and the bijection $\Phi$.

\begin{proposition}
\label{prop:include_dk}
For any integer $k\geq 1$, we have $\mathbb{D}_{\le k+1}^- = \mathbb{D}_{\le k} \sqcup \mathbb{FD}_{= k+1}$ where $\sqcup$ is the disjoint union.
\end{proposition}

We just need to determine the generating functions for $\mathbb{D}_{\leq k+1}^-$ and 
$\mathbb{FD}_{=k+1}$.

\begin{proposition}
\label{prop:gf_dkplus}
For any $k\geq 1$, the generating function for $\mathbb{D}_{\le k}^-$ 
is
$$
z^2 \frac{F_{k+1}}{F_{k+2}-z F_{k}},
$$
where $z$ takes into account the semi-perimeter.
\end{proposition}

\begin{proof}
The generating function for the cut is $\mathcal{G}(e,s,z_e,z_s) = \frac{z_e\,z_s}{1-(s+e)}$.
The generating function for the trees of the forests $F_e$ and $F_s$ is $\mathcal{T}_{\leq k} = \frac{z F_{k}}{F_{k+1}}$.
So the desired generating function is $\mathcal{G}( \mathcal{T}_{\leq k}, \mathcal{T}_{\leq k}, z, z )$.
\end{proof}

\begin{proposition}
\label{prop:gf_dkmoins}
For any $k \geq 1$, the generating function for $\mathbb{FD}_{=k+1}$ according to the semi-perimeter is
$$
z^2
\left(\frac{z^{k+1}}{F_{2k+3}}\right)^2
\frac{
	F_{k+2}
}{
	( F_{k+3} - z F_{k+1} )
}. 
$$
\end{proposition}

\begin{proof}
We will detail the 3 steps of our method:

\noindent \textbf{Step 1:}
Let $D$ be a flat directed polyomino with degree of convexity equal to $k$. 
Let us describe the trees of $F_e$ and $F_s$ of $P_D$.
From Lemma~\ref{lem:rectangle} we know that $R_{P_D}$ is a non empty rectangle and 
$I_k = E$.
Let $R$ (resp. $C$) be the row (resp. column) containing the cell $I_{k-1}$.
By construction of $\Phi$, the roots of $F_e$ (resp. $F_s$) are associated 
with the cells of the topmost (resp. rightmost) row (resp. column) of $P_D$.
Hence, the height of the trees depends on the position of their roots:

\begin{itemize}
	\item if the root is on the left (resp. below) of $C$ (resp. $R$) the height of the tree is less than or equal to $k$;
	\item if the root belongs to $C$ (resp. $R$), the height is exactly equal to $k+1$ (the two bounce 
	      paths have $k+1$ changes of direction and are the image through $\phi^{-1}$ of chains $d_r$ and $d_u$
	      of $F_e$ and $F_s$);
	\item if the root is on the right (resp. above) of $C$ (resp. $R$), the height is less than equal to $k+1$.
\end{itemize}

Let us now give a characterization of the cuts.
In the cuts, we will label the east steps by $e_1$ (resp. $e_2$) if they are mapped onto trees with height less than or equal to 
$k$ (resp. $k+1$). Similarly, we will label the south steps by $s_1$ (resp. $s_2$) if they are mapped onto trees with height less than or equal to $k$ (resp. $k+1$). Moreover, the first south (resp. last east) step of the cut
is labelled $s_2$ (resp. $e_2$). Since these two steps are not mapped onto a tree, at the end
of the process, we will obtain a bad generating function $\mathcal{G}'$ for the cut. The correct generating function is then obtained by multiplying $\mathcal{G}'$ by $\frac{z_e z_s}{e_2 s_2}$. We can use this trick because all the cuts have 
at least one south step and one east step.
Since $D$ is flat, the cut contains two special steps $u_e$ and $u_s$ in 
the columns $C$ and $R$, respectively. As $D$ is exactly $k+1$-convex,
from Corollary \ref{cor:k} we have that  
$R_{P_D} \cap D$ is non empty.
We also observe that in the cut $u_e$ precedes $u_s$.
Let $p_l$, $\tau$ and  $p_r$ be 3 paths such that 
$p_l \cdot u_e \cdot \tau \cdot u_s \cdot p_r$ is the cut $\lambda$.
\begin{figure}
\centering
\includegraphics[width=.5\linewidth]{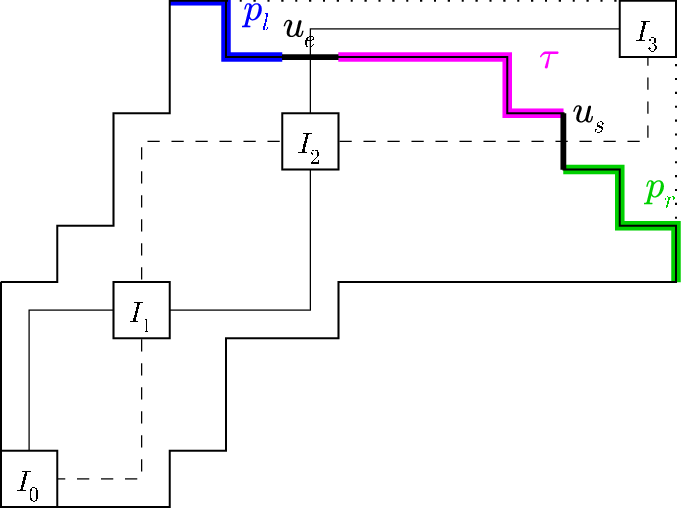} 
\caption{A directed convex polyomino $D$ where the paths $p_l$, $\tau$ and $p_r$ of $\lambda_D$ are highlighted.}\label{fig:fd}
\end{figure}
The path $p_l$ can be empty or it starts with an east step followed by any 
sequence of east/south steps.
The east steps are mapped onto trees having height less then or equal to $k$, so they are labelled by $e_1$.
The south steps are mapped onto trees having height less then or equal to $k+1$, so they are labelled by $s_2$. 
So, an unambiguous regular expression describing $p_l$ is $1 + e_1 (e_1 + s_2)^*$.
For similar argument, a regular expression for $p_r$ is 
$1+s_1 (s_1+e_2)^*$.
Consider now $\tau$: all the south and east steps are mapped onto trees with  
height less then or equal to $k+1$, so $\tau$ is a word in $e_2$ and $s_2$. 
Previously we have seen that there is at least one cell 
in $R_{P_D} \cap D$, so $\tau$ should contain the sub-word $e_2\,s_2$. An unambiguous regular expression for $\tau$ is 
$s_2^* e_2 (e_2+s_2)^* s_2 e_2^*$.
Now we just have to remark that the steps $u_e$ and $u_s$ are mapped onto 
trees having an height exactly equal to $k+1$. 

\smallskip

\noindent \textbf{Step 2:}
The generating function $\mathcal{G}(e_1, s_1, e_2, s_2, u_e, u_s, z_e. z_s)$ for the cut
$p_l \cdot u_e \cdot \tau \cdot u_s \cdot p_r$  is 
$$
\frac{z_e z_s}{e_2 s_2}
\,
\left[
	\left( 1 + \frac{e_1}{1 - e_1 - s_2} \right)
	\,
	u_e
	\,
	\frac{e_2 s_2}{ (1 - e_2) \cdot (1 - e_2 - s_2 ) \cdot (1 - s_2)}
	\,
	u_s
	\,
	\left( 1 + \frac{s_1}{1 - s_1 - e_2} \right)
\right]
,
$$
which is equal to
$
\frac{1}{1-e_1-s_2} \, \frac{z_e z_s u_e u_s}{1-e_2-s_2} \, \frac{1}{1-s_1-e_2}.
$
The generating function for the trees associated with $s_1$ and $e_1$ (resp. $s_2$ and $e_2$) is $\mathcal{T}_{\le k}$ (resp. $\mathcal{T}_{\le k+1}$), and 
the generating function for the trees associated with $u_e$ and $u_s$ is $\mathcal{T}_{= k+1}$.

\smallskip

\noindent \textbf{Step 3:}
The final generating function is 
$
\mathcal{G}(
	\mathcal{T}_{\le k}, \mathcal{T}_{\le k}, 
	\mathcal{T}_{\le k+1}, \mathcal{T}_{\le k+1}, 
	\mathcal{T}_{=k+1}, \mathcal{T}_{=k+1}, 
	z, z 
)
$ and is equal to \\
$
\left(
	\frac{\mathcal{T}_{=k+1}}{1 - \mathcal{T}_{\le k} - \mathcal{T}_{\le k+1}}
\right)^2
\,
\frac{
	z^2
}{
	1 - 2 \mathcal{T}_{\le k+1}
}
=
z^2
\,
\left(\frac{z^{k+1}}{F_{2k+3}}\right)^2
\,
\frac{
	F_{k+2}
}{
	( F_{k+3} - z F_{k+1} )
}
.
$
\end{proof}

Before calculating the generating function of directed $k$-convex polyominoes,
we need the following lemma.

\begin{lemma}\label{lem:propF}
Let $k \ge 1$ and $m \ge 1$, then the following property hold :  
$
F_{k+m-1} = F_k F_m - z F_{k-1} F_{m-1}.
$
\end{lemma}
\begin{proof}
We prove this property by induction.
\begin{description}
	\item[basis:] For $m=1$, as $F_1 = 1$ and $F_0=0$, we obtain $F_{k+1-1} = F_k.F_1 - z F_{k-1}.F_{0}$.
	
	For $m=2$, as $F_2=1$, we obtain $F_{k+2-1} = F_k - z F_{k-1} = F_k F_2 - z F_{k-1} F_{1}$.

	\item[inductive step:] We suppose that it holds for $h\leq m$ and we prove that it also holds for $m+1$.
	$$
	\begin{array}{rcl}
 		F_{k+(m+1)-1}
	&
		=
	&
		F_{k+m-1} - z F_{k+m-2}
	\\
	&
		=
	&
		( F_k F_m - z F_{k-1} F_{m-1} )    -   z ( F_{k} F_{m-1} -z F_{k-1} F_{m-2} )
	\\
	&
		=
	&
		F_k ( F_m - z F_{m-1} ) - z F_{k-1} ( F_{m-1} - z F_{m-2} )
	\\
	&
		=
	&
		F_k F_{m+1} - z F_{k-1} F_{m}.
	\end{array}
	$$ 
\end{description}
\end{proof}

\begin{figure}[htb]
\begin{center}
\includegraphics[width=.45\textwidth]{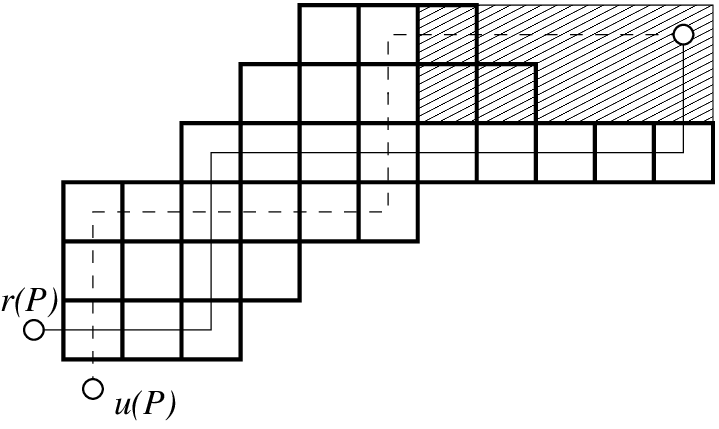}
\end{center}
\caption{A directed convex polyomino $D$, where the rectangle $R_{P_D}$ has been highlighted. \label{fig:case1}}
\end{figure}


\begin{proposition}
\label{prop:gen_k_dcp}
For $k=0$ the generating function for directed $k$-convex polyominoes with respect the semi-perimeter is  $$\mathcal{D}_0=\frac{z^2\, (1+z)}{1-z},$$ while for any $k\geq 1$, the generating function for directed $k$-convex polyominoes according to the semi-perimeter is  
$$
\mathcal{D}_{\leq k}=z^2\,  \left( \frac{F_{k+2}}{F_{2k+3}} \right)^2\,  F_{2k+2}.
$$
\end{proposition}

\begin{proof}
We just observe that directed $0$-convex polyominoes are precisely vertical and horizontal bars. For $k\geq1$, the result can be obtained as the difference of the generating functions of $\mathbb{D}_{\leq k+1}^-$ and $\mathbb{FD}_{=k+1}.$We know that $D_k = \mathbb{D}_{k+1}^+ \setminus \mathbb{D}_{k+1}^-$.
The generating function $\mathbb{D}_{\leq k+1}^-$ is 
$z^2\, \frac{F_{k+2}}{F_{k+3}-z F_{k+1}}$ (cf. Proposition~\ref{prop:gf_dkplus}) and the generating function $\mathbb{FD}_{=k+1}$ is $
z^2\, \left(\frac{z^{k+1}}{F_{2k+3}}\right)^2\,
\frac{
	F_{k+2}
}{
	( F_{k+3} - z F_{k+1} )
}
$
(cf. Proposition~\ref{prop:gf_dkmoins}).
So the generating function $\mathcal{D}_{\leq k}$ is equal to 
$$
z^2 \,\left(
	1 - 
	\left(
		\frac{ z^{k+1} }{ F_{2k+3} }
	\right)^2
\right)\,
\frac{
	F_{k+2}
}{
	( F_{k+3} - z F_{k+1} )
}
=
z^2\,
\frac{
	(F_{2k+3})^2 - z^{2k+2}
}{
	(F_{2k+3})^2
}\,
\frac{ F_{k+2} }{ F_{k+3}-zF_{k+1} }.
$$
From Proposition~\ref{prop:two_way_fk} we know that 
$
(F_k)^2 - F_{k+1}F_{k-1}
=
z^{k-1}
$
and so we obtain:
\begin{equation}
\label{eq:calculus}
z^2\,
\frac{
	F_{2k+4} F_{2k+2}
}{
	(F_{2k+3})^2
}\,
\frac{ F_{k+2} }{ F_{k+3}-zF_{k+1} }.
\end{equation}
Using Lemma~\ref{lem:propF} we can express $F_{2k+4}$ in the following way:
$$
F_{2k+4} = 
F_{(k+2)+(k+3)-1} = 
F_{k+2}F_{k+3} - zF_{k+1}F_{k+2} = 
F_{k+2} ( F_{k+3}-zF_{k+1} ).
$$
So we conclude that (\ref{eq:calculus}) is equal to
$$
	z^2\,
	\frac{
		F_{k+2} \cancel{( F_{k+3} - zF_{k+1} )} F_{2k+2}
	}{
		(F_{2k+3})^2
	}\,
	\frac{F_{k+2}}{\cancel{F_{k+3}-zF_{k+1}}}
	=
	z^2\,
	\left(
		\frac{ F_{k+2} }{ F_{2k+3} }
	\right)^2\,
	F_{2k+2}.
$$
\end{proof}

\subsection{Asymptotics}
\label{sec:asymptotic_behavior}

It is now interesting to study some facts about the asymptotic behavior of the sequence $(d_{n,\leq k})_{n,k}$ of directed $k$-convex polyominoes with semi-perimeter $n$.

\begin{proposition}
The number $d_{n,\leq k}$ of directed $k$-convex polyominoes with semi-perimeter $n$ grows like 
$$d_{n,\leq k} \sim C_k\cdot \left ( 4\cos^2(\pi/2k+3) \right ) ^n\cdot n \, .$$
\end{proposition}

\begin{proof}
Since $\mathcal{D}_{\leq k}(z)=z^2\,\frac{(F_{k+2})^2\,(F_{2k+2})}{(F_{2k+3})^2}$ is a rational function, it is known that the asymptotic form of the coefficients is
$$d_{n,\leq k}=[z^n]\mathcal{D}_{\leq k}(z)\sim C_k\cdot\mu(k)^n\cdot n$$
where $C_k$ is a $k$-dependent constant, $\mu(k)$ is given by $1/d_{2k+3}$, where $d_{2k+3}$ is the smallest real root of $F_{2k+3}$. The fact that there is a double pole in $\mathcal{D}_{\leq k}$ is responsible for the factor $n$.
In \cite{BruKnuRic} the authors observe that the roots of $F_k(z)=0$ are $(4\cos^2(j\pi/k))^{-1}$, for $1\leq j<k/2.$
In particular the reciprocals of the roots of $F_{2k+3}(z)=0$ are $(4\cos^2(j\pi/2k+3))$, for $1\leq j\leq k+1$. With basic calculus one can easily prove that the biggest reciprocal occurs for $j=1$, and so $\mu(k)=4\cos^2(\pi/2k+3)$.
\end{proof}

\begin{proposition}
Let $\mathcal{D}_{\leq k}$ and $\mathcal{D}$ be the generating functions of the directed $k$-convex polyominoes and of directed convex polyominoes with respect to the semi-perimeter, respectively. Then we have 
$$\lim_{k\rightarrow +\infty}\mathcal{D}_{\leq k}=\mathcal{D}.$$
\end{proposition}

\begin{proof}
Let $\mathcal{C}(z)=\frac{1-\sqrt{1-4z}}{2z}$ be the generating function of Catalan numbers. We can prove, by induction on $k\geq 0$, that
\begin{equation}\label{eq}
F_k=\frac{1-(z\,\mathcal{C}^2(z))^k}{\mathcal{C}^k(z)\sqrt{1-4z}}.
\end{equation}
Using the equivalence (\ref{eq}) we can write 
$$\mathcal{D}_{\leq k}=z^2\,\left( \frac{1-(z\,\mathcal{C}^2(z))^{k+2}}{1-(z\,\mathcal{C}^2(z))^{2k+3}}\right)^2\left( \frac{1-(z\,\mathcal{C}^2(z))^{2(k+1)}}{\sqrt{1-4z}}\right).$$
Clearly, the following statement holds
$$\lim_{k\rightarrow +\infty} \left( \frac{1-(z\,\mathcal{C}^2(z))^{k+2}}{1-(z\,\mathcal{C}^2(z))^{2k+3}}\right)^2=1.$$
Moreover, we have that $0<(z\,\mathcal{C}^2)^2<1$ in the domain of $\mathcal{D}_{\leq k}$, then 
$$\lim_{k\rightarrow +\infty} (z\,\mathcal{C}^2(z))^{2(k+1)}=0.$$
So we conclude that
$$\lim_{k\rightarrow +\infty}\mathcal{D}_{\leq k}=\frac{z^2}{\sqrt{1-4z}}=\mathcal{D}.$$

\end{proof}

\section{Conclusions and further work}
\label{sec:perspective}

In this paper we present a general method to calculate generating
functions for different families of directed convex polyominoes
with respect to several statistics, including the degree of convexity.
This allows to solve the problem of enumerating directed $k$-convex
polyominoes, which was presented in \cite{BatFedRinSoc} as an open
problem.

Our idea is that, for any statistic on directed convex polyominoes
that can be read on
the associated forests or on the cut, our method can suitably be
applied to obtain a
generating function according to these parameters.
We believe that our method can be applied in several different enumeration problems related with directed convex polyominoes, mainly because in Proposition~\ref{prop:directed} the sets of trees and the cut are not constrained.
Moreover, Proposition~\ref{prop:directed} can be used to write out efficient
algorithms to generate directed convex polyominoes with different
combinations of
fixed constraints. As a matter of fact, we used our method to implement a code
\footnote{
An implementation of $k$-directed polyominoes in \texttt{Sage}:
\url{http://trac.sagemath.org/ticket/17178}
.
}
in the open-source software \texttt{Sage}~\cite{Sage-Combinat, sage} for the
enumeration of directed convex polyominoes.
The code will be soon available in a future sage release.

There are several guidelines for further research. One is to study the
probability
distributions of the convexity degree for directed convex polyominoes with
a fixed semi-perimeter. Moreover, directed convex polyominoes can be
seen as special kind of tree-like tableaux
and the cut represents the states of the PASEP
(cf.~\cite{avbona13} for tree-like tableaux context and definitions).
So, our aim is to generalize Proposition~\ref{prop:directed}
for tree-like tableaux.
In tree-like tableaux $T$ associated with directed convex polyominoes, the area
counts the number of some patterns in the permutation associated with $T$.
We can try to obtain some generating function counting the area statistic.

\section*{Acknowledgements}                                                               
\label{sec:ack}

The authors are grateful to Tony Guttmann and Mireille Bousquet-Mélou for 
their help in the writing process of this article.

This research was driven by computer exploration using the open-source
mathematical software \texttt{Sage}~\cite{sage} and its algebraic
combinatorics features developed by the \texttt{Sage-Combinat}
community~\cite{Sage-Combinat}.

\nocite{*}                                                                      
\bibliographystyle{abbrvnat}

\end{document}